\documentclass[10pt]{amsart}
\usepackage{amsmath, amssymb, amsthm, mathtools}
\usepackage[latin1]{inputenc} 

\newtheorem{mydef}{Definition} 
\newtheorem{thm}{Theorem} 

\newtheorem*{main}{Main Lemma}
\newtheorem{prop}{Proposition}

\newtheorem{corr}{Corollary} 
\newtheorem{lemma}{Lemma}
\theoremstyle{remark}
\newtheorem{rem}{Remark}
\newtheorem{question}{Question}

\newtheorem{example}{Example} 

\numberwithin{equation}{section}

\DeclareMathOperator{\ord}{ord}
\DeclareMathOperator{\mult}{mult}
\DeclareMathOperator{\ind}{index}
\DeclareMathOperator{\wideg}{wideg}

\DeclareMathOperator{\Z}{\mathbb{Z}}
\DeclareMathOperator{\F}{\mathbb{F}}
\DeclareMathOperator{\N}{\mathbb{N}}
\DeclareMathOperator{\C}{\mathbb{C}}
\DeclareMathOperator{\Q}{\mathbb{Q}}
\DeclareMathOperator{\frob}{Frob}
\DeclareMathOperator{\K}{\mathbb{K}}
\DeclareMathOperator{\resit}{r\acute{e}sit}
\DeclareMathOperator{\MultiInd}{\iota}

\newcommand{\partn}[1]{{\smallskip \noindent \textbf{#1.}}}
\renewcommand{\=}{\coloneqq}
\newcommand{\dd}{\hspace{1pt}\operatorname{d}\hspace{-1pt}}

\newcommand{\al}{\widehat{\alpha}}
\newcommand{\be}{\widehat{\beta}}
\newcommand{\ga}{\widehat{\gamma}}
\newcommand{\cgamma}{\check{\gamma}}

\title{Residue fixed point index and wildly ramified power series}

\begin{document}

\author{Jonas Nordqvist}
\address{Department of Mathematics, Linnaeus University, V{\"a}xj{\"o}, Sweden}
\email{jonas.nordqvist@lnu.se}

\author{Juan Rivera-Letelier}
\address{Department of Mathematics, University of Rochester. Hylan Building, Rochester, NY~14627, U.S.A.}
\email{riveraletelier@gmail.com}
\urladdr{http://rivera-letelier.org/}

\begin{abstract}
  In this paper, we study power series having a fixed point of multiplier~$1$.
  First, we give a closed formula for the residue fixed point index, in terms of the first coefficients of the power series.
  Then, we use this formula to study wildly ramified power series in positive characteristic.
  Among power series having a multiple fixed point of small multiplicity, we characterize those having the smallest possible lower ramification numbers in terms of the residue fixed point index.
  Furthermore, we show that these power series form a generic set, and, in the case of convergent power series, we also give an optimal lower bound for the distance to other periodic points.
\end{abstract}


\maketitle

\section{Introduction}

Consider an open subset~$U$ of~$\C$ and a holomorphic map~$f \colon U \to \C$.
For a fixed point~$z_0$ of~$f$, the derivative~$f'(z_0)$ is invariant under coordinate changes.
In the case~$z_0$ is isolated as a fixed point of~$f$, a related invariant is defined by the countour integral
\begin{equation}
  \label{eq:19}
  \ind(f, z_0)
  \=
  \frac{1}{2\pi i} \oint \frac{\dd z}{z - f(z)},
\end{equation}
where we integrate on a sufficiently small simple closed curve around~$z_0$ that is positively oriented.
The complex number~\eqref{eq:19} is invariant under coordinate changes and is called the \emph{residue fixed point index of~$f$ at~$z_0$}.
Together with the related holomorphic fixed point formula, it is one of the basic tools in complex dynamics, see, \emph{e.g.}, \cite[\S12]{Mil06c} for background, and \cite{BuffEpstein2002, Buff03, BuffXavEcalle2013} for some results where the residue fixed point index plays an important r{\^o}le.
See also~\cite[Exercise~5.10]{Silverman2007} for an extension to an arbitrary ground field.

In the case~$f'(z_0) \neq 1$, a direct computation shows that~\eqref{eq:19} is equal to~$\frac{1}{1 - f'(z_0)}$.
We give a closed formula for~\eqref{eq:19} in the case~$f'(z_0) = 1$, in terms of the first coefficients of the power series expansion of~$f$ about~$z_0$ (Theorem~\ref{thm:closed-formula} in~\S\ref{sec:closed-formula}).
This formula holds for an arbitrary ground field.
We also show that the residue fixed point index is invariant under coordinate changes, and use it to study normal forms.
We also study the behavior of the residue fixed point under iteration.

In our succeeding results, we restrict to ground fields of positive characteristic and power series having the origin as a fixed point of multiplier~$1$.
Such power series are called \emph{wildly ramified}.\footnote{This terminology arises from the study of field automorphisms.
  Every power series~$f$ with coefficients in a field~$\K$ that satisfies~$f(0) = 0$ and~$f'(0) = 1$, defines a field automorphism of~$\K[[t]]$ given by~$g \mapsto g \circ f$.
  When~$\K$ is of positive characteristic, this type of field automorphism is traditionally known as \emph{wildly ramified}, due to the behavior of its associated ramification numbers.}
See, \emph{e.g.}, \cite{Sen1969,Keating1992,LaubieSaine1998,Win04} for background on wildly ramified power series, \cite{KallalKirkpatrick2019,LaubieMovahhediSalinier2002,LindahlNordqvist2018,LindahlRiveraLetelier2013,LindahlRiveraLetelier2015,Fransson2017,RiveraLetelier2003} for results related to this paper, and \cite{HermanYoccoz1983, Lindahl2004, LindahlZieve2010, Ruggiero2015} and references therein for local dynamics of analytic germs in positive characteristic.
See also, \emph{e.g.}, \cite{johnson1988,Camina2000} and references therein, for the myriad of group-theoretic results about the ``Nottingham group'', which is the group under composition formed by the wildly ramified power series.

Every wildly ramified power series has associated a sequence of ``lower ramification'' numbers.
It encodes the multiplicity of the origin for the iterates of the power series.
We study the lower ramification numbers of power series for which the multiplicity at the origin is small.
First, we characterize those power series having the smallest possible lower ramification numbers.
They are characterized by the nonvanishing of {\'E}calle's ``iterative residue'', which is a dynamical version of the residue fixed point index (Theorem~\ref{thm:q-ramification} in~\S\ref{sec:q-ramification}).
As a consequence, we obtain that these power series form a generic set.
In the case of convergent power series, we also give an optimal lower bound for the distance to other periodic points (Theorem~\ref{thm:lower-bound} in~\S\ref{sec:lower-bound}).
This gives an affirmative solution to~\cite[Conjecture~1.2]{LindahlRiveraLetelier2013}, for generic multiple fixed points of a fixed and small multiplicity, and to~\cite[Conjecture~4.3]{KallalKirkpatrick2019}.

We proceed to describe our results more precisely.

\subsection{Closed formula for the residue fixed point index}
\label{sec:closed-formula}
Our first result is a closed formula for the residue fixed point index of a fixed point of multiplier~$1$.
We allow an arbitrary ground field, and an arbitrary power series about a fixed point.
In particular, we allow non-convergent power series.
To simplify the notation, throughout the rest of the paper we restrict to the case of a power series~$f$ fixing the origin, and denote~$\ind(f, 0)$ by~$\ind(f)$.

\begin{mydef}
  Let~$\K$ be a field and~$f$ a power series with coefficients in~$\K$ satisfying~$f(0)=0$ and~$f(z)\neq z$.
  The \emph{residue fixed point index of~$f$ at~$0$}, denoted by~$\ind(f)$, is the coefficient of~$\frac{1}{z}$ in the Laurent series expansion about~$0$ of
  \[ \frac{1}{z-f(z)}. \]
\end{mydef}

Clearly, this definition agrees with~\eqref{eq:19} in the case where~$\K = \C$, $z_0 = 0$, and~$f$ is holomorphic on a neighborhood of~$0$.

To state our first result, denote by $\mathbb{N}$ the set of nonnegative integers and for an integer~$q \ge 1$ and~$(\MultiInd_0, \ldots, \MultiInd_q)$ in~$\N^{q + 1}$, define
\begin{displaymath}
  |(\MultiInd_0, \ldots, \MultiInd_q)| \= \sum_{j = 0}^{q} \MultiInd_j
  \text{ and }
  \| (\MultiInd_0, \ldots, \MultiInd_q) \| \= \sum_{j = 1}^q j \MultiInd_j.
\end{displaymath}

\begin{thm}[Residue fixed point index formula]
  \label{thm:closed-formula}
  Let $\K$ be a field, $q\geq 1$ an integer, and~$f$ a power series with coefficients in~$\K$ of the form
  \begin{equation}
    \label{psform}f(z)
    =
    z\left(1 + \sum_{j=q}^{+\infty} a_jz^j\right), \text{ with } a_q \neq 0.
  \end{equation}
  Then we have
  \begin{equation}
    \label{eq:closed-formula}
    \ind(f)
    =
    - \frac{1}{a_q^{q+1}}\sum_{\substack{\MultiInd \in \N^{q+1} \\ |\MultiInd| = q, \| \MultiInd \| = q }} (-1)^{q-\MultiInd_0}\binom{q-\MultiInd_0}{\MultiInd_{1},\ldots,\MultiInd_{q}}\prod_{j = 0}^q a_{q + j}^{\iota_j}.
  \end{equation}
\end{thm}

We also show that the residue fixed point index is invariant under coordinate changes (Proposition~\ref{conj} in~\S\ref{invarresidue}) and use the residue fixed point index to study normal forms (Proposition~\ref{nf} in~\S\ref{sec:nf}).
Both of these results, together with Theorem~\ref{thm:closed-formula}, are used to prove our results below.
In Appendix~\ref{app:resit} we use Theorem~\ref{thm:closed-formula} to study the behavior under iterations of the residue fixed point index, and of the closely related ``iterative residue'' defined below.

\subsection{Wildly ramified power series}
\label{sec:q-ramification}
Let~$\K$ be a field, and~$f$ a power series with coefficients in~$\K$ such that~$f(0) = 0$ and~$f(z) \neq z$.
The \emph{multiplicity of~$0$ as a fixed point of~$f$} is the lowest degree of a nonzero term in $f(z) - z$.
We denote it by~$\mult(f)$.

From now on we assume the characteristic~$p$ of~$\K$ is positive.
The power series~$f$ is \emph{wildly ramified} if~$\mult(f) \ge 2$, or equivalently, if~$0$ is a multiple fixed point of~$f$.
Note that~$f$ is wildly ramified if and only if~$f'(0) = 1$.
For a wildly ramified power series~$f$, the \emph{lower ramification numbers~$\{ i_n(f) \}_{n = 0}^{+\infty}$ of~$f$} are defined by
\[ i_n(f) \= \mult(f^{p^n})-1. \]
See, \emph{e.g.}, \cite{Sen1969,Keating1992,LaubieSaine1998,Win04} and references therein for background on wildly ramified power series and their lower ramification numbers.
Due to their relation to ultrametric dynamics, they have been studied in, \emph{e.g.}, \cite[\S3.2]{RiveraLetelier2003}, \cite{LindahlRiveraLetelier2015,LindahlRiveraLetelier2013,LindahlNordqvist2018}.
Note that the lower ramification numbers are invariant under coordinate changes.

If we put
\begin{displaymath}
  q \= \mult(f) - 1 \ge 1,
\end{displaymath}
then the results of Sen in~\cite{Sen1969} imply that, in the case~$q \le p - 1$, for every integer~$n \ge 0$ we have
\begin{equation}
  \label{eq:q-minimal}
  i_n(f) \ge q (1+p+ \cdots +p^n),  
\end{equation}
see Proposition~\ref{prop:qramifisminramif} in~\S\ref{sec:lowramif}.
Following~\cite{Fransson2017}, for an integer~$q \ge 1$ that is not divisible by~$p$, we say that~$f$ is \emph{$q$-ramified} if equality holds in~\eqref{eq:q-minimal} for every~$n$.
In the case~$q = 1$, $1$-ramified power series are also known as ``minimally ramified'' \cite{LaubieMovahhediSalinier2002,LindahlRiveraLetelier2013,LindahlRiveraLetelier2015}.
$q$-Ramified power series appear naturally as reductions of invertible elements of formal groups, see for example~\cite[\emph{Proposition}~4.2]{LaubieMovahhediSalinier2002} for the case~$q = 1$, and~\cite[\emph{Corollaire}~3.12]{LaubieMovahhediSalinier2002} for general~$q$ not divisible by~$p$. 
Note that when~$q$ is divisible by~$p$, for every~$n \ge 1$ we have~$i_n(f) = i_0(f)p^n$ \cite{Sen1969}, so we cannot have equality in~\eqref{eq:q-minimal}.

Our next result characterizes $q$-ramified power series when~$q \le p - 1$, and shows that $q$-ramified power series are generic among power series having the origin as a fixed point of multiplicity~$q + 1$.
We restrict to odd~$p$, as the case~$p = 2$ is treated in~\cite{LindahlRiveraLetelier2015,LindahlRiveraLetelier2013}.
As in~\cite[Theorem~E]{LindahlRiveraLetelier2013}, our characterization is best stated in terms of the ``iterative residue'', which is a dynamical variant of the residue fixed point index introduced by {\'E}calle in the complex setting.
For a power series~$f$ satisfying ${f(0) = 0}$ and ${f(z) \neq z}$, the \emph{iterative residue of~$f$} is defined by\footnote{We keep {\'E}calle's notation ``$\resit$'', an abbreviation of the French ``\emph{r{\'e}sidue it{\'e}ratif}''.} 
\begin{equation}
  \label{eq:17}
  \resit(f)
  \=
  \frac{1}{2}\mult(f) - \ind(f).
\end{equation}
See, \emph{e.g.}, \cite[\S{I}]{Ecalle1975},  or~\cite[\S12]{Mil06c} for background on the iterative residue.

\begin{thm}[$q$-ramified power series]
  \label{thm:q-ramification}
  Let~$p$ be an odd prime number and $\K$ a field of characteristic~$p$.
  Furthermore, let~$q$ be in~$\{1, \ldots, p - 1 \}$, and let~$f$ be a power series with coefficients in~$\K$ satisfying~$\mult(f) = q+1$.
  Then~$f$ is $q$-ramified if and only if $\resit(f) \neq 0$.
\end{thm}

Let~$q \ge 1$ be an integer, $x_q$, $x_{q + 1}$, \ldots indeterminates over~$\K$, and consider the generic power series
\begin{displaymath}
  f(\zeta) \= \zeta \left( 1 + \sum_{j = q}^{+ \infty} x_j \zeta^j \right).
\end{displaymath}
Then by Theorem~\ref{thm:closed-formula}, $x_q^{q + 1} \resit(f)$ is equal to

\begin{equation}
  \label{eq:genericity polynomial}
  \left( \frac{q + 1}{2} \right) x_q^{q + 1} + \sum_{\substack{\MultiInd \in \N^{q+1} \\ |\MultiInd| = q, \| \MultiInd \| = q }} (-1)^{q-\MultiInd_0}\binom{q-\MultiInd_0}{\MultiInd_{1},\ldots,\MultiInd_{q}}\prod_{j = 0}^q x_{q + j}^{\iota_j},
\end{equation}
which is a polynomial in~$x_q$, $x_{q + 1}$, \ldots, $x_{2q}$ with coefficients in~$\F_p$.\footnote{Note that this polynomial is isobaric of degree~$q(q + 1)$.}
Thus, the following corollary is a direct consequence of Theorem~\ref{thm:q-ramification}.

\begin{corr}
  \label{c:genericity}
  Let~$p$ be an odd prime number, $\K$ a field of characteristic~$p$, and~$q$ in~$\{1, \ldots, p - 1 \}$.
  Then, among power series with coefficients in~$\K$ for which the origin is a fixed point of multiplicity~$q + 1$, those that are $q$-ramified are generic.
\end{corr}

The following corollary is essentially a reformulation of the previous corollary in terms of the \emph{Nottingham group~$\mathcal{N}(\K)$}, which is the group under composition formed by all wildly ramified power series with coefficients in~$\K$.
Since the work of Johnson~\cite{johnson1988}, this group has been extensively studied for its interesting group-theoretic properties.
See for instance the survey article~\cite{Camina2000}.

Given an integer~$q \ge 1$, consider the subgroup of~$\mathcal{N}(\K)$,
\begin{displaymath}
  \mathcal{N}_q(\K)
  \=
  \{ f \text{ power series with coefficients in~$\K$ satisfying~$\mult(f) \ge q + 1$} \}.
\end{displaymath}
Note that in the case~$q = 1$, we have~$\mathcal{N}_1(\K) = \mathcal{N}(\K)$.

\begin{corr}
  Let~$p$ be an odd prime number, $\K$ a field of characteristic~$p$, and~$q$ in~$\{1, \ldots, p - 1 \}$.
  Then, an element~$f$ of~$\mathcal{N}_q(\K)$ is $q$-ramified if and only if~$\resit(f) \neq 0$. 
  In particular, $q$-ramified power series are generic in~$\mathcal{N}_q(\K)$.
\end{corr}

This answers~\cite[Question~1.4]{KallalKirkpatrick2019} for~$q$ in~$\{1, \ldots, p - 1\}$.

In the case $q=1$, Theorem~\ref{thm:q-ramification} was shown by Lindahl and the second named author~\cite[Theorem~E]{LindahlRiveraLetelier2013}.
This last result also applies to the case~$p=2$, and asserts that a power series of the form~\eqref{psform} with~$q = 1$ is $1$-ramified if and only if
\begin{displaymath}
  \resit(f) \neq 0
  \text{ and }
  \resit(f) \neq 1.
\end{displaymath}
In the case $q=2$, Theorem~\ref{thm:q-ramification} was shown by the first named author~\cite[Theorem~1]{Fransson2017}, with~$\resit(f)$ replaced by~\eqref{eq:genericity polynomial}.
In the case~$q = 3$ and~$\K = \F_p$, Theorem~\ref{thm:q-ramification} was shown by Kallal and Kirkpatrick in the first version of~\cite{KallalKirkpatrick2019}, with~$\resit(f)$ replaced by~\eqref{eq:genericity polynomial}.
  After a preliminary version of this paper was completed, we received a new version of~\cite{KallalKirkpatrick2019} proving Theorem~\ref{thm:q-ramification} when restricted to those~$q$ satisfying~$q^2 < p$, and with~$\resit(f)$ replaced by~\eqref{eq:genericity polynomial}.

  Theorem~\ref{thm:q-ramification} and its corollaries are not expected to extend to the case~$q \ge p + 1$ not divisible by~$p$.
In fact, we give examples showing that the conclusion of Theorem~\ref{thm:q-ramification} is false for~$q = p + 1$, see Example~\ref{ex:p + 1} in~\S\ref{sec:furth-results-exampl}.
About genericity, if~$q \ge p + 1$ is not divisible by~$p$, then the results of Laubie and Sa{\"{\i}}ne in~\cite{LaubieSaine1998} imply that the inequality~\eqref{eq:q-minimal} fails in general, even for~$n = 1$.
Thus, for~$q \ge p + 1$ the $q$-ramified power series are not expected to be generic among power series having~$0$ as a fixed point of multiplicity~$q + 1$.
So, the following question arises naturally.

\begin{question}
  Let~$p$ be a prime number, $\K$ a field of characteristic~$p$, and~$q \ge p + 1$ an integer that is not divisible by~$p$.
  How are the lower ramification numbers of a generic power series in~$\mathcal{N}_q(\K)$? \footnote{Recently, the first named author answered this question completely in~\cite{Nor1909}.}

\end{question}

In the case~$q = p + 1$, it seems that for a generic power series satisfying~$\mult(f) = q + 1$, we have for every~$n \ge 0$
\begin{displaymath}
  i_n(f) = 1 + p + \cdots + p^{n + 1}.
\end{displaymath}
See also Example~\ref{ex:p + 1} in~\S\ref{sec:furth-results-exampl}, and the discussion following it.
\subsection{Periodic points of wildly ramified power series}
\label{sec:lower-bound}
Our next result is about the distribution of periodic points of a convergent $q$-ramified power series.
To state it, we introduce some notation.
Given an ultrametric field $(\K,|\cdot|)$, denote by
\begin{displaymath}
  \mathcal{O}_{\K}
  \=
  \{ \zeta \in \K : |\zeta| \le 1 \},
  \text{ and }
  \mathfrak{m}_{\K}
  \=
  \{ \zeta \in \K : |\zeta| < 1 \},
\end{displaymath}
the ring of integers of~$\K$ and the maximal ideal of~$\mathcal{O}_{\K}$, respectively.

\begin{thm}[Periodic points lower bound]
  \label{thm:lower-bound}
  Let~$p$ be an odd prime number, let~$q$ be in~$\{1, \ldots, p - 1 \}$, and let~$(\K, |\cdot|)$ be an ultrametric field of characteristic~$p$.
  Furthermore, let~$f$ be a power series with coefficients in~$\mathcal{O}_{\K}$ of the form
  \[f(\zeta)
    \equiv
    \zeta(1 + a\zeta^q) \mod \langle \zeta^{q+2} \rangle, \text{ with } a\neq0.\]
  Then, for every fixed point~$\zeta_0$ of~$f$ in~$\mathcal{O}_{\K}$ that is different from~$0$ we have $|\zeta_0|\geq |a|$, and for every periodic point~$\zeta_0$ of~$f$ in~$\mathcal{O}_{\K}$ that is not a fixed point, we have
  \begin{equation}\label{normbound}
    |\zeta_0|
    \ge
    |a| \cdot |\resit(f)|^\frac{1}{p}.
  \end{equation}
\end{thm}

We give explicit examples for which equality holds in~\eqref{normbound} for every periodic point that is not fixed, when~$q \le p - 3$ (Example~\ref{e:optimality} in~\S\ref{sec:furth-results-exampl}).
  We recall that by Theorem~\ref{thm:closed-formula} we can explicitly compute~$\resit(f)$, see also~\eqref{eq:genericity polynomial}, so the lower bound in Theorem~\ref{thm:lower-bound} is effective.
  Note also that the lower bound given by Theorem~\ref{thm:lower-bound} is trivial in the case that~$f$ is not $q$-ramified, because by Theorem~\ref{thm:q-ramification} we have $\resit(f) = 0$ in this case.

Note that every convergent power series about~$0$ without constant term is conjugated to a power series with coefficients in~$\mathcal{O}_{\K}$ by a scale change.
So, the following corollary is a direct consequence of Theorem~\ref{thm:lower-bound}.

\begin{corr}
  Let~$\K$ be an ultrametric field of positive characteristic, and let~$q \ge 1$ be an integer that is strictly smaller than the characteristic of~$\K$.
  Moreover, let~$f$ be a $q$-ramified power series with coefficients in~$\K$ that converges on a neighborhood of the origin.
  Then the origin is isolated as a periodic point of~$f$.
\end{corr}

Combined with Corollary~\ref{c:genericity} and~\cite[Theorem~E with $p = 2$]{LindahlRiveraLetelier2013}, the previous corollary implies the following result as a direct consequence.

\begin{corr}
  \label{c:generic isolation}
  Let~$p$ be a prime number and fix~$m$ in~$\{2, \ldots, p \}$.
  Then, over a field of characteristic~$p$, a generic fixed point of multiplicity~$m$ is isolated as a periodic point.
\end{corr}

This corollary solves~\cite[Conjecture~1.2]{LindahlRiveraLetelier2013} in the affirmative, for generic multiple fixed points of a fixed and small multiplicity, as well as~\cite[Conjecture~4.3]{KallalKirkpatrick2019}.
In the case~$m = 2$, Corollary~\ref{c:generic isolation} is~\cite[Main Theorem]{LindahlRiveraLetelier2015}.

In the case~$q = 1$, Theorem~\ref{thm:lower-bound} was shown by Lindahl and the second named author~\cite[Theorem~B]{LindahlRiveraLetelier2015}.
This last result also applies to~$p = 2$.
In the case~${q = 2}$, and for power series with integer coefficients, Theorem~\ref{thm:lower-bound} was shown by Lindahl and the first named author~\cite[Theorem~A]{LindahlNordqvist2018}.

\subsection{Organization}
\label{sec:organziation}

In~\S\ref{sec:invariance} and in Appendix~\ref{app:resit}, we study the residue fixed point index over a field of arbitrary characteristic.
Theorem~\ref{thm:closed-formula} is shown in~\S\ref{sec:proof-of-closed-formula}, the invariance of the residue fixed point index under coordinate changes is shown in~\S\ref{invarresidue}, and in~\S\ref{sec:nf} we study normal forms.
All these results are used in the in the proof of Theorems~\ref{thm:q-ramification} and~\ref{thm:lower-bound}.
In Appendix~\ref{app:resit}, we study the behavior under iterations of the iterative residue.

In~\S\ref{shortproof} we give a short proof of Theorem~\ref{thm:q-ramification} that relies on a result of Laubie and Sa{\"{\i}}ne in~\cite{LaubieSaine1998}.
After some preliminaries on lower ramification numbers in~\S\ref{sec:lowramif}, this proof is given in~\S\ref{sec:proof of ramification}.

In~\S\ref{s:periodic points} we give a self-contained proof of Theorem~\ref{thm:q-ramification}, and the proof of Theorem~\ref{thm:lower-bound}.
We obtain both of these from our main technical result that we state as the ``Main Lemma'' at the beginning of~\S\ref{s:periodic points}.
The proof of this result occupies~\S\ref{s:proof of Main Lemma}.
In~\S\ref{sec:self-contained-proof}, we use the Main Lemma and the results in~\S\ref{sec:invariance} to obtain more information about the coefficients of the iterates of a wildly ramified power series as in Theorem~\ref{thm:q-ramification}.
This is stated as Proposition~\ref{prop:deltaiterates}, and it implies Theorem~\ref{thm:q-ramification} as a direct consequence.
It is also the main new ingredient in the proof of Theorem~\ref{thm:lower-bound}, which is given in~\S\ref{sec:lowerbound}.

In~\S\ref{sec:furth-results-exampl}, we gather several examples illustrating our results.

\subsection*{Acknowledgments}
We would like to thank the referees for their valuable comments and corrections that helped improve the exposition of the paper.

The first named author acknowledges support from \emph{Kungliga Vetenskapsakademien}, grant MG2018-0011, for his visit to the second named author at University of Rochester. He would also like to thank the second named author for his hospitality and for providing an excellent working environment during said visit. Finally, the first named author would also like to thank his supervisor Karl-Olof Lindahl for fruitful discussions in the early stages of this project.

The second named author acknowledges partial support from NSF grant DMS-1700291.

\section{The residue fixed point index}\label{sec:invariance}
In this section we prove the closed formula (Theorem~\ref{thm:closed-formula}) and the invariance under coordinate changes of the residue fixed point index.
The former is proved in~\S\ref{sec:proof-of-closed-formula}, and the latter is stated and proved in~\S\ref{invarresidue}.
In~\S\ref{sec:nf} we also use the residue fixed point index to study normal forms of wildly ramified power series.

Given a ring~$R$ and elements $a_1, \ldots, a_n$ of~$R$, denote by $\langle a_1,\ldots,a_n \rangle$ the ideal generated by~$a_1, \ldots, a_n$.
Furthermore, denote by~$R[[z]]$ the ring of power series with coefficients in~$R$ in the variable~$z$, and denote by~$\ord_z$ the $z$-adic valuation on~$R[[z]]$, \emph{i.e.}, for a nonzero~$f$ in~$R[[z]]$ the valuation~$\ord_z(f)$ is the unique integer~$j$ such that~$f$ is in~$z^jR[[z]] \setminus z^{j+1}R[[z]]$, and for~$f = 0$ we have $\ord_z(0) \= +\infty$.

\subsection{Closed formula for the residue fixed point index}
\label{sec:proof-of-closed-formula}
In this section we prove Theorem~\ref{thm:closed-formula}, after the following lemma.

\begin{lemma}\label{sumlittlea}
  Let $\K$ be a field, $q\geq 1$ an integer, and~$f$ a power series with coefficients in~$\K$ of the form~\eqref{psform}.
  Then~$- a_q^{q + 1} \ind(f)$ is equal to the coefficient of~$z^{q}$ in
  \begin{equation}
    \label{eq:sumlittlea}
    \sum_{r=0}^qa_q^r(-1)^{q-r}(a_{q+1}z+\cdots+a_{2q}z^{q})^{q-r}.
  \end{equation}
\end{lemma}

\begin{proof}
From the definition, $\ind(f)$ is equal to the coefficient of~$\frac{1}{z}$ in the Laurent series expansion about~$0$ of
\begin{equation}
  \label{eq:laurentindf}
  \begin{split}
    \frac{1}{z-f(z)}
  & =
  -\frac{1}{a_qz^{q+1}+a_{q+1}z^{q+2}+\cdots+a_{2q}z^{2q+1}+\cdots}
  \\ & =
  -\frac{1}{a_qz^{q+1}}\cdot\frac{1}{1 + \frac{a_{q+1}}{a_q}z +\frac{a_{q+2}}{a_q}z^2+\cdots}
  \\ & =
  -\frac{1}{a_q^{q+1}z^{q+1}}\sum_{j=0}^{+ \infty}a_q^{q-j}(-1)^j\left(a_{q+1}z +a_{q+2}z^2+\cdots\right)^j.
  \end{split}
\end{equation}
Thus, $\ind(f)$ is equal to the coefficient of~$z^q$ in the sum in \eqref{eq:laurentindf}.
Note that for $k\geq 2q+1$, the coefficient $a_k$ does not contribute to the coefficient of~$z^q$ in the sum in \eqref{eq:laurentindf}.
Also for $j>q$, the corresponding term in the sum in~\eqref{eq:laurentindf} has no term in~$z^q$.
Hence, $\ind(f)$ is equal to the coefficient of~$z^q$ in~\eqref{eq:sumlittlea}, as claimed.
\end{proof}

\begin{proof}[Proof of Theorem~\ref{thm:closed-formula}]
  In view of Lemma~\ref{sumlittlea}, it is sufficient to compute the coefficient of~$z^q$ in~\eqref{eq:sumlittlea}.
  Using the multinomial theorem and regrouping, \eqref{eq:sumlittlea} is equal to
\begin{multline*}
    \sum_{r=0}^qa_q^{r}(-1)^{q-r}\sum_{\substack{(\MultiInd_1, \ldots, \MultiInd_q) \in \N^q \\ \MultiInd_{1}+\ldots+\MultiInd_{q} =q-r}}\binom{q-r}{\MultiInd_{1},\ldots, \MultiInd_{q}}\prod_{j=1}^{q} (a_{q+j} z^j)^{i_{j}}
    \\ =
    \sum_{\substack{\MultiInd \in \N^{q + 1} \\ |\MultiInd| = q}}(-1)^{q-\MultiInd_0}\binom{q-\MultiInd_0}{\MultiInd_{1},\ldots,\MultiInd_{q}} \left(\prod_{j = 0}^q a_{q + j}^{\iota_j}\right)z^{\|\iota\|}.
\end{multline*}
In the last expression, the term in~$z^{q}$ is given by restricting the sum to those multi-indices~$\MultiInd$ satisfying~$\| \MultiInd \| = q$.
This proves the theorem.
\end{proof}

\subsection{The residue fixed point index is invariant}\label{invarresidue}

This section is devoted to prove the following proposition.
\begin{prop}
  \label{conj}
  Let~$\K$ be a field.
  Then, among power series~$f$ with coefficients in~$\K$ and satisfying~$f(0)=0$ and~$f(z) \neq z$, the residue fixed point index is invariant under coordinate changes.
  That is, for every power series~$\varphi$ with coefficients in~$\K$ such that~$\varphi(0) = 0$ and~$\varphi'(0) \neq 0$, the power series~$\widehat{f} \= \varphi \circ f \circ \varphi^{-1}$ satisfies
  \begin{displaymath}
    \ind(\widehat{f}) = \ind(f).
  \end{displaymath}
\end{prop}

The proof of this proposition is given after the following lemma.

\begin{lemma}\label{powersarezero}
  Let~$\K$ be a field and~$\varphi$ a power series with coefficients in~$\K$ such that~$\varphi(0) = 0$ and~$\varphi'(0) \neq 0$.
  Then for every integer $N\geq 1$, the coefficient of~$\frac{1}{z}$ in the Laurent series expansion about~$0$ of
  \[\frac{\varphi'(z)}{\varphi(z)^{N+1}}\]
  is zero.
\end{lemma}

\begin{proof}
  Put~$\varphi(z) = {\displaystyle \sum_{j=0}^{+\infty} a_jz^j}$ and for a field automorphism~$\sigma$ of~$\K$ put
  \[\varphi^\sigma(z) \= \sum_{j=0}^{+\infty} \sigma(a_j)z^j.\]

  If the characteristic of~$\K$ is zero or if the characteristic of~$\K$ is positive and it does not divide~$N$, then the lemma is clear as
  \[\frac{\varphi'(z)}{\varphi(z)^{N+1}} = \left(-\frac{1}{N}\cdot \frac{1}{\varphi(z)^N}\right)'. \]
  So we assume~$\K$ is of characteristic $p>0$ and that~$N$ is divisible by~$p$.
  Let $\ell\geq 1$ be the largest integer such that $p^\ell \mid N$, and put $n\=p^{-\ell}N$.
  Moreover, denote by $\frob\colon \K\to\K$ the Frobenius automorphism, given by $\frob(z) \= z^p$, and put $\sigma\=\frob^\ell$.
  Then we have 
\begin{equation}\label{eq:frob}
\frac{\varphi'(z)}{\varphi(z)^{N +1}} 
=
\frac{(\varphi^\sigma)'(z^{p^\ell})}{\varphi^\sigma(z^{p^\ell})^{n+1}}
\cdot \left( \frac{\varphi^\sigma(z^{p^\ell})}{(\varphi^\sigma)'(z^{p^\ell})} \cdot \frac{\varphi'(z)}{\varphi(z)} \right).
\end{equation}
Since $n$ is not divisible by $p$, the coefficient of~$\frac{1}{z}$ in the Laurent series expansion about~$0$ of $\frac{(\varphi^{\sigma})'(z)}{(\varphi^{\sigma}(z))^{n+1}}$ is zero.
So the coefficient of~$\frac{1}{z^{p^\ell}}$ in the Laurent series expansion about~$0$ of~$\frac{(\varphi^\sigma)'(z^{p^\ell})}{\varphi^\sigma(z^{p^\ell})^{n+1}}$ is zero.
Together with
\[\ord_z\left(\frac{\varphi^{\sigma}(z^{p^\ell})}{(\varphi^\sigma)'(z^{p^\ell})}\cdot \frac{\varphi'(z)}{\varphi(z)}\right) = p^\ell-1,\]
this implies that the coefficient of~$\frac{1}{z}$ in the Laurent series expansion about~$0$ of~$\frac{\varphi'(z)}{\varphi(z)^{N +1}}$ is zero, which is the desired assertion.
\end{proof}

\begin{proof}[Proof of Proposition~\ref{conj}]
  If $f'(0) \neq 1$, then~$\ind(f)$ is equal to $\frac{1}{1-f'(0)}$, which is easily seen to be invariant under coordinate changes.
  Assume $f'(0) = 1$, and put
  \begin{displaymath}
    \Delta(z) \= f(z)-z
    \text{ and }
    q\=\ord_z(\Delta(z))-1.
  \end{displaymath}
  Our hypothesis $f(z)\neq z$ implies that $q$ is finite and our assumption $f'(0)=1$ implies that $q\geq 1$.

Let~$\varphi$ be a power series with coefficients in~$\K$ such that $\varphi(0)=0$ and $\varphi'(0) \neq 0$, and put
\begin{displaymath}
  \widehat{f}\=\varphi^{-1}\circ f \circ \varphi
  \text{ and }
  \widehat{\Delta}(z) \= \widehat{f}(z)-z.
\end{displaymath}
Clearly $\widehat{f}'(0) = 1$, so $\ord_z(\widehat{\Delta}(z))\geq 2$.
Moreover,
\begin{equation}
  \label{DeltaH}
  \begin{split}
    \Delta \circ \varphi(z)
    & =
    \varphi(\widehat{f}(z)) - \varphi(z)
    \\ & =
    \varphi(z + \widehat{\Delta}(z)) - \varphi(z)
    \\ & \equiv
    \varphi'(z)\widehat{\Delta}(z) \mod \langle \widehat{\Delta}(z)^2\rangle.    
  \end{split}
\end{equation}
Since $\ord_z(\Delta) = q+1$ and $\ord_z(\varphi') = 0$, we conclude that
\begin{displaymath}
  \ord_z(\Delta\circ \varphi) = q+1
  \text{ and }
  {\ord_z(\varphi'\cdot \widehat{\Delta})} = \ord_z(\widehat{\Delta}).
\end{displaymath}
On the other hand, by (\ref{DeltaH}) we have $\ord_z(\Delta\circ\varphi-\varphi'\cdot\widehat{\Delta})\geq 2\ord_z(\widehat{\Delta})$ and therefore
\begin{displaymath}
  \ord_z(\widehat{\Delta}) = \ord_z(\Delta \circ \varphi) = q+1.
\end{displaymath}
Using (\ref{DeltaH}) again we obtain
\[\Delta \circ \varphi \equiv \varphi'\cdot \widehat{\Delta} + \langle z^{2q+2} \rangle,\]
and conclude that~$\ind(\widehat{f})$ is equal to the coefficient of~$\frac{1}{z}$ in the Laurent series expansion about~$0$ of
\[\frac{\varphi'}{\Delta\circ \varphi}. \]

Putting
\[\left(\frac{1}{\Delta}\right)(z) \= \sum_{i = -(q+1)}^{+\infty} a_i z^i,\]
we have
\begin{displaymath}
  \left(\frac{\varphi'}{\Delta\circ \varphi}\right)(z)
=
\sum_{N = 0}^{q} a_{- (N + 1)} \frac{\varphi'(z)}{\varphi(z)^{N + 1}}
+
\sum_{i = 0}^{+\infty} a_i \varphi(z)^i \varphi'(z).
\end{displaymath}
By Lemma~\ref{powersarezero}, the coefficient of~$\frac{1}{z}$ in the Laurent series expansion about~$0$ of the right-hand side is equal to that of~$a_{-1} \frac{\varphi'(z)}{\varphi(z)}$, which is clearly equal to $a_{-1}$.
This completes the proof of the proposition.
\end{proof}

\subsection{Normal forms in positive characteristic}\label{sec:nf}
Let~$\K$ be a field and~$f$ a power series with coefficients in~$\K$ such that~$q \= \mult(f) - 1$ is finite and satisfies~$q \ge 1$.
In the case of $\K=\C$, or more generally if~$\K$ is of characteristic zero, there exists a (formal) power series conjugating~$f$ to the polynomial
\begin{equation}
  \label{eq:22}
  z(1+z^q + \ind(f)z^{2q}).  
\end{equation}
When~$\K$ is of characteristic zero, this polynomial is called the \emph{normal form of~$f$}.

This statement is false if~$\K$ is of positive characteristic.
Our goal in this section is to prove the following proposition giving a sufficient condition for~$f$ to have the same normal form up to a high order.

\begin{prop}\label{nf}
  Let~$p$ be a prime number and~$\K$ a field of characteristic~$p$.
  Moreover, let $q$ be in~$\{1, \ldots, p - 1 \}$, and let~$f$ be a power series with coefficients in~$\K$ satisfying~$\mult(f) = q + 1$.
  Then, $f$ is conjugated to a power series with coefficients in a finite extension of~$\K$, of the form
  \begin{equation}
    \label{eq:nff}
    z(1 + z^q + \ind(f)z^{2q})\mod \langle z^{2q + p + 1} \rangle.
  \end{equation}
\end{prop}

The proof of this proposition is given after the following lemma.

\begin{lemma}\label{removeterms}
  Let~$\K$ be a field, $q \ge 1$ an integer, and~$f$ a power series with coefficients in~$\K$ of the form
  \begin{displaymath}
    f(z)
    =
    z\left(1 + \sum_{j=q}^{+\infty} a_jz^j\right), \text{ with } a_q \neq 0.
  \end{displaymath}
  Then, for every integer~$k \ge 1$ such that~$a_{q + k} \neq 0$ and~$k\neq q$ in $\K$, there is~$c$ in~$\K$ such that for the polynomial~$\varphi(z) \= z(1 + c z^k)$, we have
  \[ \varphi \circ f \circ \varphi^{-1}(z) \equiv z(1 + a_qz^q + \cdots + a_{q+k-1}z^{q+k-1}) \mod \langle z^{q+k+2}\rangle.\]
\end{lemma}

\begin{proof}
  Let~$c$ be a constant in~$\K$ to be chosen later, and put
  \[ \varphi(z) \= z(1 + cz^k)
    \text{ and }
    \widehat{f}(z)
    \=
    \varphi \circ f \circ \varphi^{-1}(z)
    = z \left( 1 + \sum_{j = q}^{+ \infty} \widehat{a}_jz^j \right).\]
  Then we find
  \begin{displaymath}
    \begin{split}
      \varphi \circ f(z)
      & \equiv
      z (1 + a_qz^q + \cdots + a_{q+k}z^{q+k})(1 + cz^k(1 + a_qz^q)^k)
      \mod \langle z^{q+k+2} \rangle
      \\ & \equiv
      z(1 + cz^k + a_qz^q + \cdots + a_{q+k-1}z^{q+k-1}
      \\&\qquad
      + ((k+1)ca_q + a_{q+k})z^{q+k}) \mod \langle z^{q+k+2} \rangle,      
    \end{split}
  \end{displaymath}
  and
\begin{displaymath}
  \begin{split}
    \widehat{f} \circ \varphi(z)
    & \equiv
    z (1 + cz^k)(1 + \widehat{a}_qz^q (1 + cz^k)^q + \widehat{a}_{q + 1} z^{q + 1} + \cdots + \widehat{a}_{q+k}z^{q+k})
    \\ &
    \quad \mod \langle z^{q+k+2} \rangle
    \\ & \equiv
  z(1 + cz^k + \widehat{a}_qz^q +\cdots + \widehat{a}_{q+k-1}z^{q+k-1}
  \\&\qquad
  + ((q+1)c\widehat{a}_q + \widehat{a}_{q + k})z^{q+k}) \mod \langle z^{q+k+2} \rangle.
  \end{split}
\end{displaymath}
Equating both expression yields
\[a_q = \widehat{a}_q,\ldots, a_{q+k-1} = \widehat{a}_{q+k-1},\]
and
\[ \widehat{a}_{q + k} = (k - q)c a_q + a_{q+k}.\]
By our assumption $k\neq q$ in~$\K$, we can take~$c = - \frac{a_{q+k}}{a_q(k-q)}$ to obtain~$\widehat{a}_{q + k} = 0$.
\end{proof}

\begin{proof}[Proof of Proposition~\ref{nf}]
  Denote by~$a \neq 0$ the coefficient of~$z^{q+1}$ in~$f$, and let~$\gamma$ in a finite extension of~$\K$ be such that~$\gamma^q = a^{-1}$.
  Note that the power series~$\widehat{f}(z) \= \gamma^{-1} f(\gamma z)$ satisfies~$\mult(\widehat{f}) = q + 1$ and that the coefficient of~$z^{q + 1}$ in~$\widehat{f}$ is equal to~$1$.

  Since by assumption~$q$ is in~$\{1,\ldots, p - 1\}$, we can apply Lemma~\ref{removeterms} successively with~$k = 1, \ldots, q - 1$, to obtain that there is a polynomial~$\varphi$ with coefficients in~$\K[\gamma]$, such that~$\varphi(0) = 0$, $\varphi'(0) = 1$, and
  \begin{displaymath}
    g(z)
    \=
    \varphi \circ \widehat{f} \circ \varphi^{-1}(z)
    \equiv
    z (1 + z^q) \mod \langle z^{2q + 1} \rangle.
  \end{displaymath}
  Note that by Theorem~\ref{thm:closed-formula} the coefficient of~$z^{2q + 1}$ in~$g$ is equal to~$\ind(g)$ and by Proposition~\ref{conj} we have~$\ind(g) = \ind(\widehat{f}) = \ind(f)$.
  Thus,
  \begin{displaymath}
    g(z)
    \equiv
    z (1 + z^q + \ind(f) z^{2q}) \mod \langle z^{2q + 2} \rangle.
  \end{displaymath}
  Finally, we apply Lemma~\ref{removeterms} successively with~$k = q + 1, \ldots, q + p - 1$, to obtain that there is a polynomial~$\phi$ with coefficients in~$\K[\gamma]$, such that~$\phi(0) = 0$, $\phi'(0) = 1$, and
  \begin{displaymath}
    \phi \circ g \circ \phi^{-1}(z)
    \equiv
    z (1 + z^q + \ind(f) z^{2q}) \mod \langle z^{2q + p + 1} \rangle.
    \qedhere
  \end{displaymath}
\end{proof}

\section{$q$-Ramified power series}\label{shortproof}
After some preliminaries on lower ramification numbers in~\S\ref{sec:lowramif}, in~\S\ref{sec:proof of ramification} we give a short proof of Theorem~\ref{thm:q-ramification} that relies on a result of Laubie and Sa{\"{\i}}ne in~\cite{LaubieSaine1998}.
See~\S\ref{sec:self-contained-proof} for a self-contained proof of Theorem~\ref{thm:q-ramification}.

\subsection{Lower ramification numbers}\label{sec:lowramif}
In this section we fix a prime number~$p$ and a field~$\K$ of characteristic~$p$.
Recall that for a power series~$f$ in~$\K[[\zeta]]$ and an integer~$n \ge 1$, the lower ramification number~$i_n(f)$ of~$f$ is
\[i_n(f) = \mult(f^{p^n})-1.\]
Lower ramification numbers have been studied by several authors, \emph{e.g.}, \cite{Sen1969,Keating1992,LaubieSaine1998,LaubieMovahhediSalinier2002}.
    A central theorem of Sen~\cite[Theorem~1]{Sen1969} states that if for some $n\geq0$ we have $i_n(f) < +\infty$, then
\[i_n(f) \equiv i_{n-1}(f) \pmod{p^n}.\]
The following consequence of Sen's theorem shows that for~$q$ in~$\{1, \ldots, p - 1 \}$, a $q$-ramified power series can be thought of as minimal in the sense that for every integer~$n$ the lower ramification number~$i_n(f)$ is least possible.

\begin{prop}\label{prop:qramifisminramif}
  Let~$p$ be a prime number and~$\K$ a field of characteristic $p$.
  Then for every~$q$ in $\{1,\ldots,p-1\}$, and every power series~$f$ in~$\K[[\zeta]]$ satisfying $\mult(f) = q+1$, we have for every integer~$n \ge 1$
  \begin{equation}
    \label{ineq}
    i_n(f) \geq q(1 + p + \cdots + p^n).
  \end{equation}
\end{prop}

The proof of this proposition is given after the following lemma.
To state this lemma, we introduce some notation.
Let~$R$ be a ring, and~$f$ a power series in~$R[[z]]$ of the form~$f(z) \equiv z \mod \langle z^2 \rangle$.
Following~\cite[\emph{Exemple}~3.19]{RiveraLetelier2003} and~\cite{LindahlRiveraLetelier2013}, define recursively for every integer $m\geq 0$ the power series~$\Delta_m$ by
\begin{equation}
  \label{eq:27}
  \Delta_0(z) \= z,  
\end{equation}
and for~$m \ge 1$ by
\begin{equation}
  \label{eq:23}
  \Delta_m(z)
  \=
  \Delta_{m-1}(f(z)) - \Delta_{m-1}(z).
\end{equation}
If~$R$ is of characteristic zero, then for every prime number~$p$ a direct computation shows that we have
\begin{equation}
  \label{eq:7}
  \Delta_p(z) \equiv f^p(z)-z \mod \langle p \rangle.
\end{equation}
In the case~$R$ is of characteristic~$p$, we have~$\Delta_p(z) = f^p(z)-z$.

\begin{lemma}\label{lemma:delta}
  Let~$p$ be a prime number and~$\K$ a field of characteristic $p$.
  Given a wildly ramified power series~$f$ in~$\K[[\zeta]]$, let~$( \Delta_m )_{m = 0}^{+\infty}$ be as above.
  Then for every integer~$m \ge 1$ we have
  \begin{equation}
    \label{eq:24}
    \ord_\zeta(\Delta_m)-\ord_\zeta(\Delta_{m-1})
    \geq
    \ord_\zeta(\Delta_1)-1.
  \end{equation}
\end{lemma}

\begin{proof}
  Put~$q \= {\ord_\zeta(\Delta_1)-1}$, $f(\zeta) = \zeta\left(1 + \sum_{i = q}^{+ \infty} b_i\zeta^i\right)$, $r \= \ord_{\zeta}(\Delta_m)$, and~$\Delta_m(\zeta)=\sum_{i = r}^{+ \infty} a_i\zeta^i$.
  Then
  \[\Delta_{m+1}(\zeta)
    =
    \sum_{i =r}^{+ \infty} a_i\zeta^i\left[(1+b_q\zeta^q+\cdots)^i - 1\right], \]
  and therefore~$\ord_\zeta(\Delta_{m+1}) \ge r+q$.
\end{proof}

\begin{proof}[Proof of Proposition \ref{prop:qramifisminramif}]
  We prove~\eqref{ineq} by induction in~$n$.
  To prove~\eqref{ineq} for~$n = 1$, let~$(\Delta_m)_{m = 0}^{+ \infty}$ be as in~\eqref{eq:27} and~\eqref{eq:23}.
  Then for every integer~$m \ge 1$ we have $\ord_\zeta(\Delta_m) - \ord_\zeta(\Delta_{m-1}) \geq q$ by Lemma~\ref{lemma:delta}.
  An induction argument combined with~\eqref{eq:7} gives
  \begin{displaymath}
    i_1(f) = \ord_\zeta(\Delta_p) - 1 \geq qp = p i_0(f).
  \end{displaymath}
  But by Sen's theorem we have $i_1(f) \equiv i_0(f) \pmod{p}$, so
  \begin{equation}
    \label{firstinduc}
    i_1(f) \geq qp + q.
  \end{equation}
  This proves~\eqref{ineq} for~$n = 1$.
  
  Let~$n \ge 1$ be an integer for which~\eqref{ineq} holds, and put $g(\zeta) \= f^{p^n}(\zeta)$.
  Let~$(\widehat{\Delta}_m)_{m = 0}^{+ \infty}$ be the sequence~$(\Delta_m)_{m = 0}^{+ \infty}$ given by~\eqref{eq:27} and~\eqref{eq:23} with~$f$ replaced by~$g$.
  Then by Lemma~\ref{lemma:delta} for every integer~$m \ge 1$ we have
\[\ord_\zeta(\widehat{\Delta}_m)-\ord_\zeta(\widehat{\Delta}_{m-1})
  \geq
  \ord_{\zeta}(\widehat{\Delta}_1)
  =
  i_0(g). \]
An induction argument together with~\eqref{eq:7}, implies
\begin{equation}
  \label{eq:26}
  i_{n + 1}(f)
  =
  i_1(g)
  =
  \ord_\zeta(\widehat{\Delta}_p) - 1
  \geq
  p i_0(g)
  =
  p i_n(f).
\end{equation}
If the inequality in our induction assumption~\eqref{ineq} is strict, then we have
\begin{displaymath}
  i_{n + 1}(f)
  \ge
  p + p q (1 + p + \cdots + p^n)
  >
  q(1 + p + \cdots + p^{n + 1}).
\end{displaymath}
If equality holds in~\eqref{ineq}, then by Sen's theorem we have
\begin{displaymath}
  i_{n + 1}(f) \equiv q(1 + p + \cdots + p^n) \pmod{p^{n+1}}.
\end{displaymath}
Combined with~\eqref{eq:26}, this implies
\begin{displaymath}
  i_{n+1}(f)
  \geq
  q + p q (1 + p + \cdots + p^n)
  =
  q(1 + p + \cdots + p^{n + 1}).
\end{displaymath}
In all the cases we obtain~\eqref{ineq} with~$n$ replaced by~$n + 1$.
This completes the proof of the induction step, and of the the proposition.
\end{proof}

\subsection{Proof of Theorem~\ref{thm:q-ramification}}
\label{sec:proof of ramification}
In the proof of Theorem~\ref{thm:q-ramification} we use the following result of Laubie and Sa{\"{\i}}ne.
\begin{prop}[\cite{LaubieSaine1998}, Corollary~1]
  \label{laubiecorr}
Let~$p$ be a prime number, $\K$ a field of characteristic~$p$, and~$f$ in~$\K[[\zeta]]$ such that $f(0)=0$ and $f'(0)=1$.
If
\begin{displaymath}
  p\nmid i_0(f)
  \text{ and }
  i_1(f) < (p^2-p+1)i_0(f), 
\end{displaymath}
then for every integer $n\geq 1$ we have
\[ i_n(f) = i_0(f) + (1 + p + \cdots + p^n)(i_1(f)-i_0(f)). \]
\end{prop}

In view of this result, the proof of Theorem~\ref{thm:q-ramification} reduces to show that for~$q$ in~$\{1, \ldots, p - 1 \}$ and~$f$ in~$\K[[\zeta]]$ satisfying~$i_0(f) = q$, the conditions
\begin{displaymath}
  i_1(f) = q (p + 1)
  \text{ and }
  \resit(f) \neq 0
\end{displaymath}
are equivalent.
The following is the key ingredient, together with Proposition~\ref{nf} and the invariance of the residue fixed point index under coordinate changes shown in~\S\ref{sec:invariance}.

\begin{prop}\label{prop:deltaiteratesshort}
Let~$p$ be an odd prime number and consider the rings
\begin{displaymath}
  \Z_{(p)} \= \left\{\frac{m}{n} \in \Q : m, n \in \Z, p\nmid n \right\},
  \end{displaymath}
  \begin{displaymath}
      F_1 \= \Z_{(p)}[x_0, x_1],
  \text { and }
  F_\infty \= \Z_{(p)}[x_0, x_1,x_2,\ldots].
\end{displaymath}
Then for each integer~$q \ge 1$ not divisible by~$p$, the power series~$\widehat{f}$ in~$F_\infty[[\zeta]]$ defined by
\[ \widehat{f}(\zeta)
  \=
  \zeta\left(1 + x_0 \zeta^q + x_1\zeta^{2q} + \zeta^{2q}\sum_{i=1}^{+ \infty} x_{i+1}\zeta^i\right), \]
satisfies
\begin{displaymath}
  \widehat{f}^p(\zeta)
  \equiv
  \zeta \left(1 + x_0^{p - 1} \left( x_0^2 \frac{q + 1}{2} - x_1 \right) \zeta^{q(p + 1)}\right) \mod \langle p, \zeta^{q(p + 1) + 2} \rangle.
\end{displaymath}
\end{prop}

The proof of Theorem~\ref{thm:q-ramification} is given at the end of this section, after the proof of this proposition.
To prove this proposition we use the strategy introduced in~\cite[\emph{Exemple}~3.19]{RiveraLetelier2003} and~\cite{LindahlRiveraLetelier2013}, using~\eqref{eq:27} and~\eqref{eq:23}.
We also use the following elementary lemma.
\begin{lemma}\label{wilson}
  Let~$p$ be an odd prime number, $a$ and~$b$ in~$\F_p$ such that $a\neq 0$, and let~$w \colon \F_p \to \F_p$ be defined by~$w(n) \= an +b$.
  Denoting $s' \= -a^{-1}b$, we have
  \begin{displaymath}
    \prod_{s \in \F_p \setminus \{s'\}} w(s) = -1
  \text{ and }
    \sum_{s \in \F_p \setminus\{s'\}} \frac{1}{w(s)} = 0.
  \end{displaymath}
\end{lemma}
\begin{proof}
  We use the fact that the nonconstant affine map~$w$ is a bijection of~$\F_p$.
  Together with Wilson's theorem this implies the first assertion.
  The second assertion follows from the fact that, since~$p$ is odd, the sum of all nonzero elements in $\F_p$ is~$0$.
\end{proof}

\begin{proof}[Proof of Proposition~\ref{prop:deltaiteratesshort}]
  Let~$(\Delta_m)_{m = 0}^{+ \infty}$ be given by~\eqref{eq:27} and~\eqref{eq:23}.
  For each integer~$m \ge 1$ define~$\alpha_m$, and $\beta_m$ in the ring~$F_1 \= \Z_{(p)}[x_0, x_1]$ by the recursive relations 
\begin{align}
  \alpha_{m+1} & \= x_0 (qm+1)\alpha_m, \label{receq1short}
  \\ 
  \beta_{m+1} & \= \left[x_0^2 \binom{qm+1}{2} + x_1(qm+1)\right]\alpha_m + x_0 (q(m+1)+1)\beta_m,
                \label{receq2short}
\end{align}
with initial conditions $\alpha_1 \= x_0$ and $\beta_1 \= x_1$.
We prove by induction that for every integer~$m \ge 1$ we have
\begin{equation}
  \label{deltaclaimshort}
  \Delta_m(\zeta)
  \equiv
  \alpha_m\zeta^{qm+1} + \beta_m\zeta^{q(m+1)+1} \mod \langle \zeta^{q(m + 1) + 2} \rangle.
\end{equation}
For $m=1$ this holds by definition. Assume further that it is valid for some $m\geq1$.
Then
\begin{displaymath}
  \begin{split}
  \Delta_{m+1}(\zeta)
  & = \Delta_m(\widehat{f}(\zeta)) - \Delta_m(\zeta)
  \\ &\equiv \alpha_m\zeta^{qm+1}\left[\left(1 + x_0 \zeta^q + x_1\zeta^{2q}  + \cdots\right)^{qm+1} - 1\right]\\
&\qquad + \beta_m\zeta^{q(m+1)+1}\left[\left(1 + x_0 \zeta^q + x_1\zeta^{2q}  + \cdots\right)^{q(m+1)+1}-1\right]
\\ & \qquad
\mod \langle \zeta^{q(m+2)+2} \rangle
\\
&\equiv \alpha_m\left[ \zeta^{q(m+1)+1}x_0 (qm+1) + \zeta^{q(m+2)+1}\left( x_0^2 \binom{qm+1}{2} + x_1 (qm+1)\right)\right] \\
&\qquad+ \beta_m\zeta^{q(m+2)+1}x_0 (q(m+1)+1) \mod \langle \zeta^{q(m+2)+2} \rangle.    
  \end{split}
\end{displaymath}
In view of~\eqref{receq1short} and~\eqref{receq2short}, this proves the induction step and~\eqref{deltaclaimshort}.

By~\eqref{eq:7} and~\eqref{deltaclaimshort}, to prove the proposition it is sufficient to prove
\begin{equation}
  \label{eq:5short}
  \alpha_p \equiv 0 \mod p F_1
  \text{ and }
  \beta_p \equiv x_0^{p - 1} \left( x_0^2 \frac{q + 1}{2} - x_1 \right) \mod p F_1.
\end{equation}
We do this by solving explicitly the linear recurrences described in~\eqref{receq1short}, and \eqref{receq2short}.
By telescoping \eqref{receq1short}, we obtain for every~$m \ge 1$ the solution
\begin{equation}
  \label{alpham}
  \alpha_m = x_0^m \prod_{j=1}^{m-1}(qj+1).
\end{equation}
Taking~$m = p$ we obtain the first congruence in~\eqref{eq:5short}.

On the other hand, inserting \eqref{alpham} in~\eqref{receq2short} yields
\[\beta_{m+1} = \left(x_0^2 \frac{qm}{2} + x_1 \right) x_0^m \prod_{j=1}^m (qj+1) + x_0 (q(m+1)+1)\beta_m.\]
Noting that for every~$j \ge 0$ we have~$qj + 1 > 0$, we utilize the substitution
\[\beta^*_m \= \beta_m \bigg/ \left( x_0^{m - 1} \prod_{j=1}^m (qj+1) \right),\]
which yields
\[\beta^*_{m+1}
  =
  \beta^*_{m} + \left(x_0^2 \frac{qm}{2} + x_1 \right)\frac{1}{q(m+1)+1}.\] 
Using $\beta^*_1 = \frac{x_1}{q + 1}$, we obtain inductively for every~$m \ge 1$
\begin{displaymath}
  \beta^*_m
  =
  \sum_{r=1}^{m} \left(x_0^2 \frac{q(r - 1)}{2} + x_1 \right)\frac{1}{qr + 1}.  
\end{displaymath}
Equivalently,
\begin{equation}\label{betam}
  \beta_m
  =
  x_0^{m - 1} \sum_{r=1}^{m}\left[ \left( x_0^2 \frac{q(r - 1)}{2} + x_1 \right) \prod_{j \in \{1, \ldots, m\} \setminus \{ r \}} (qj+1) \right].
\end{equation}
When~$m = p$ every term in the sum above contains a factor~$p$, except for the unique~$r$ in~$\{1, \ldots, p \}$ such that $qr \equiv -1 \pmod{p}$.
Denote by~$r_0$ this value of~$r$.
Then by Lemma~\ref{wilson}, we have
\begin{displaymath}
  \begin{split}
  \beta_p
  &\equiv
    x_0^{p - 1} \left(\frac{ x_0^2 q(r_0 - 1)}{2} + x_1\right) \prod_{j \in \{ 1, \ldots, p \} \setminus \{ r_0 \}} (qj+1) \mod p F_1
  \\ &\equiv
       x_0^{p - 1} \left(x_0^2 \frac{q+1}{2} - x_1\right) \mod p F_1.    
  \end{split}
\end{displaymath}
This proves the second congruence in~\eqref{eq:5short} and thus the proposition.
\end{proof}

\begin{proof}[Proof of Theorem~\ref{thm:q-ramification}]
  By Proposition~\ref{nf} and our hypothesis that~$q$ is in~$\{1, \ldots, p-1\}$, we have that~$f$ is conjugated to a power series~$g$ in~$\K[[\zeta]]$ of the form
  \[g(\zeta) \equiv \zeta(1 + \zeta^q + \ind(f)\zeta^{2q}) \mod \langle \zeta^{3q+2} \rangle.\]
  Since
  \begin{displaymath}
    i_0(g) = i_0(f) = q
    \text{ and }
    i_1(g) = i_1(f),
  \end{displaymath}
  by Proposition~\ref{laubiecorr} the series~$f$ is $q$-ramified if and only if~$i_1(g) = q (p + 1)$.
  
  Let~$\Z_{(p)}$ and~$F_\infty$ be as in Proposition~\ref{prop:deltaiteratesshort}.
  Moreover, let~$h \colon F_\infty \to \K$ be the unique ring homomorphism extending the reduction map~$\Z_{(p)} \to \F_p$, such that~$h(x_1) = \ind(f)$ and such that for every~$i \ge 2$ the element~$h(x_i)$ of~$\K$ is the coefficient of~$\zeta^{2q + i}$ in~$g$.
Then~$h$ extends to a ring homomorphism~$F_{\infty}[[ \zeta ]] \to \K [[ \zeta ]]$ that maps~$\widehat{f}$ to~$g$.
So, Proposition~\ref{prop:deltaiteratesshort} implies
\begin{displaymath}
  g^{p}(\zeta) - \zeta
\equiv
\resit(f) \zeta^{q(p+1)+1}  \mod \langle \zeta^{q(p+1)+2} \rangle.
\end{displaymath}
This proves that~$i_1(g) = q(p + 1)$ if and only if~$\resit(f) \neq 0$ and completes the proof of the theorem.
\end{proof}

\section{Periodic points of $q$-ramified power series}
\label{s:periodic points}
In this section we give a self-contained proof of Theorem~\ref{thm:q-ramification}, and the proof of Theorem~\ref{thm:lower-bound}.
In doing so, we obtain more information about the coefficients of the iterates of a wildly ramified power series as in Theorem~\ref{thm:q-ramification} (Proposition~\ref{prop:deltaiterates} in~\S\ref{sec:self-contained-proof}).
This extra information is used to prove Theorem~\ref{thm:lower-bound} in~\S\ref{sec:lowerbound}.

The main ingredients in the proofs of Theorems~\ref{thm:q-ramification} and~\ref{thm:lower-bound} are the results on the residue fixed point index in~\S\ref{sec:invariance}, and the following result that is proved in~\S\ref{s:proof of Main Lemma}.

\begin{main}
  Let~$p$ be an odd prime number, and let~$\Z_{(p)}$, $F_1$ and~$F_{\infty}$ be the rings defined in Proposition~\ref{prop:deltaiteratesshort}.
  Moreover, let~$q \ge 1$ be an integer that is not divisible by~$p$, and~$\ell \ge 1$ an integer satisfying
  \begin{displaymath}
    \ell \equiv q \pmod{p},
    \text{ and }
    \ell \le p - 1
    \text{ or }
    2\ell + 1 \le q.
  \end{displaymath}
  Then the power series~$\widehat{f}$ in~$F_\infty[[\zeta]]$ defined by
\[\widehat{f}(\zeta) \= \zeta\left(1 + x_0\zeta^{q} + x_1\zeta^{q+\ell} + \zeta^{q+2\ell}\sum_{i=1}^\infty x_{i+1}\zeta^{i}\right), \]
satisfies the following property: There are~$\beta$ and~$\gamma$ in~$F_1$ such that
\begin{align}
  \label{e:main beta}
  \beta
  & \equiv
  \begin{cases}
    x_0^{p - 1} \left( x_0^2 \frac{q + 1}{2} - x_1 \right) \mod p F_1
    & \text{if $q \le p - 1$};
    \\
    - x_0^{p - 1} x_1 \mod p F_1
    & \text{if $q \ge p + 1$},
  \end{cases}
      \\
      \label{e:main gamma}
      \gamma
  & \equiv
    \begin{cases}
    - x_0^{p - 2} \left(x_0^2 \frac{q + 1}{2} - x_1 \right)^2 \mod p F_1
    & \text{if $q \le p - 1$};
    \\
      - x_0^{p - 2} x_1^2 \mod p F_1
    & \text{if $q \ge p + 1$},
    \end{cases}
      \intertext{ and }
      \widehat{f}^p(\zeta)
  & \equiv
    \zeta \left(1 + \beta \zeta^{qp + \ell} + \gamma \zeta^{qp + 2\ell} \right) \mod \langle p, \zeta^{qp + 2\ell + 2} \rangle.
\end{align}
\end{main}

\subsection{Self-contained proof of Theorem~\ref{thm:q-ramification}}
\label{sec:self-contained-proof}

The goal of this section is to deduce the following proposition from the Main Lemma, which is a more precise version of Theorem~\ref{thm:q-ramification}.
It is also one of the main ingredients of the proof of Theorem~\ref{thm:lower-bound}, which is given in~\S\ref{sec:lowerbound}.

\begin{prop}\label{prop:deltaiterates}
Let~$p$ be an odd prime number and~$\K$ a field of characteristic~$p$.
Furthermore, let~$q$ be in~$\{1, \ldots, p-1 \}$, let~$f$ in~$\K[[\zeta]]$ be of the form
\[f(\zeta)
\equiv
\zeta(1 + a_0\zeta^q + a_1\zeta^{2q}) \mod \langle \zeta^{3q+2}\rangle, \text{ with } a_0\neq 0, \]
and for each integer~$n\geq1$, put
\begin{align*}
  \chi_n
  & \=
    a_0^{\frac{p^{n+1}-1}{p-1}}\left(\frac{q+1}{2}-\frac{a_1}{a_0^2}\right)^{\frac{p^n-1}{p-1}},
    \intertext{ and }
    \psi_n
  & \=
    -a_0^{\frac{p^{n+1}-1}{p-1}+1}\left(\frac{q+1}{2}-\frac{a_1}{a_0^2}\right)^{\frac{p^n-1}{p-1}+1}.
\end{align*}    
Then we have
\begin{displaymath}
  f^{p^n}(\zeta) - \zeta
  \equiv
  \chi_n\zeta^{q\frac{p^{n+1}-1}{p-1}+1} + \psi_n\zeta^{q\frac{p^{n+1}-1}{p-1}+q+1} \mod \langle \zeta^{q\frac{p^{n+1}-1}{p-1}+q+2} \rangle.
\end{displaymath}
In particular, $f$ is $q$-ramified if and only if
\begin{displaymath}
  \resit(f) = \frac{q + 1}{2} - \frac{a_1}{a_0^2} \neq 0.
\end{displaymath}
\end{prop}

The proof of Proposition~\ref{prop:deltaiterates} is given after the following lemma.

\begin{lemma}\label{elimination}
  Let~$p$ be an odd prime number, $q$ in~$\{1,\ldots,p-1\}$, and~$d \ge 1$ an integer satisfying $d\equiv 1 \pmod{p}$.
  Furthermore, let~$\K$ be a field of characteristic~$p$ and let~$f$ in~$\K[[\zeta]]$ be of the form
  \[f(\zeta)
    \equiv
    \zeta\left(1+ a_0\zeta^{qd} + a_1\zeta^{q(d+1)}\right)  \mod \langle \zeta^{q(d+1)+2} \rangle, \text{ with } a_0\neq 0.\]
  Then there is a polynomial~$\varphi$ with coefficients in~$\K$ such that~$\mult(\varphi) \ge {q + 2}$, and such that~$\varphi$ conjugates~$f$ to a power series~$g$ satisfying
  \begin{align}
    \label{eq:12}
    g(\zeta)
    & \equiv
      \zeta\left(1 + a_0\zeta^{qd} + a_1\zeta^{q(d+1)}\right) \mod \langle \zeta^{q(d+1) + p + 1} \rangle,
      \intertext{ and }
      \notag
      g^p(\zeta)
      & \equiv f^p(\zeta)  \mod \langle \zeta^{i_1(f)+ q + 2} \rangle.
  \end{align}
\end{lemma}
\begin{proof}
  Noting that~$qd \equiv q \pmod{p}$, we can apply Lemma~\ref{removeterms} successively with~$q$ replaced by~$qd$, and with
  \begin{displaymath}
    k
    =
    q + 1, \ldots, q + p - 1,
  \end{displaymath}
  to obtain a polynomial~$\varphi$ satisfying~$\mult(\varphi) \ge q + 2$, such that~$g \= \varphi \circ f \circ \varphi^{-1}$ satisfies~\eqref{eq:12}.

  To prove the second assertion, note that~$\varphi$ also conjugates~$f^p$ to~$g^p$, so by Lemma~\ref{removeterms}
  \begin{displaymath}
    i_1(f) = i_1(g)
    \text{ and }
    f^p(\zeta) \equiv g^p(\zeta) \mod \langle \zeta^{i_1(f) + \mult(\varphi)} \rangle.
  \end{displaymath}
  The desired assertion follows from the inequality~$\mult(\varphi) \ge q + 2$.
  This completes the proof of the lemma.
\end{proof}

\begin{proof}[Proof of Proposition~\ref{prop:deltaiterates}]
The last assertion is a direct consequence of the first and of~\eqref{eq:genericity polynomial}.
  
  To prove the first assertion, for each integer $n\geq 0$ put $d_n \= 1 + p + \cdots + p^n$, and note that
\begin{displaymath}
  d_n \equiv 1 \pmod{p},
  \text{ and }
  d_np+1 = d_{n + 1}.
\end{displaymath}
We first prove by induction that for every integer~$n \ge 0$ there are~$\chi_n$ and~$\psi_n$ in~$\K$, such that
\begin{equation}
  \label{eq:13}
  f^{p^n}(\zeta)
  \equiv
  \zeta\left(1 + \chi_{n}\zeta^{qd_n} + \psi_{n}\zeta^{q(d_n+1)}\right) \mod \langle \zeta^{q(d_n+1)+2} \rangle.
\end{equation}
The case~$n = 0$ is trivial, with
\begin{equation}
  \label{eq:10}
  \chi_0 = a_0
  \text{ and }
  \psi_0 = a_1.
\end{equation}
Let~$n \ge 0$ be a given integer, and assume the desired assertion is true for~$n$.
By Lemma~\ref{elimination} there is a power series~$g$ with coefficients in~$\K$ such that
\begin{align}
  \notag
  g(\zeta)
  & \equiv
  \zeta\left(1 + \chi_{n}\zeta^{qd_n} + \psi_{n}\zeta^{q(d_n+1)}\right) \mod \langle \zeta^{q(d_n+2)+2} \rangle,
  \intertext{ and }
  \label{eq:14}
  g^p(\zeta)
  & \equiv
  f^{p^{n + 1}}(\zeta) \mod \langle \zeta^{i_{n + 1}(f) + q + 2} \rangle.
\end{align}
  Define $\Z_{(p)}, F_1$ and $F_\infty$ as in Proposition~\ref{prop:deltaiteratesshort}.
  Moreover, let $\widehat{g}$ in~$F_\infty[[\zeta]]$ be of the form
  \[ \widehat{g}(\zeta)
    \=
    \zeta\left(1+x_0 \zeta^{qd_n} + x_1 \zeta^{q(d_n + 1)} + \zeta^{q(d_n + 2)}\sum_{j=1}^{+ \infty}x_{j+1} \zeta^j\right), \] 
  let $h\colon F_\infty \to \K$ be the unique ring homomorphism extending the reduction map $\Z_{(p)} \to \F_p$, such that $h(x_0) = \chi_n$, $h(x_1) = \psi_n$, and such that for every $i\geq 2$ the element $h(x_i)$ of $\K$ is the coefficient of $\zeta^{q(d_n+2)+i}$ in~$\widehat{g}$.
  Then~$h$ extends to a ring homomorphism $F_\infty[[\zeta]] \to \K[[\zeta]]$ that maps $\widehat{g}$ to~$g$.
  In the case~$n = 0$, note that~$\widehat{f}$ in the Main Lemma is equal to~$\widehat{g}$, so
  \begin{multline*}
    g^p(\zeta)
    \equiv
    \zeta \left( 1 + \chi_0^{p+1}\left(\frac{q+1}{2} - \frac{\psi_0}{\chi_0^2}\right)\zeta^{q(p+1)}
    \right. \\ \left.
      - \chi_0^{p+2} \left(\frac{q+1}{2}-\frac{\psi_0}{\chi_0^2}\right)^2\zeta^{q(p+2)} \right) \mod \langle \zeta^{q(p+2)+2} \rangle.
  \end{multline*}
  Together with~\eqref{eq:14} with~$n = 0$, this implies
  \begin{displaymath}
    i_1(f) = i_1(g) \ge q(p + 1) = qd_1,
  \end{displaymath}
  and~\eqref{eq:13} with~$n = 1$,
  \begin{equation}
    \label{eq:11}
    \chi_1
    \=
    \chi_0^{p+1}\left(\frac{q+1}{2} - \frac{\psi_0}{\chi_0^2} \right)
      \text{ and }
      \psi_1
      \= - \chi_0^{p+2} \left(\frac{q+1}{2}-\frac{\psi_0}{\chi_0^2}\right)^2.
    \end{equation}
    In the case~$n \ge 1$, the Main Lemma with~$q$ replaced by~$qd_n$ and~$\ell$ replaced by~$q$, implies
  \begin{displaymath}
    g^p(\zeta)
    \equiv
    \zeta \left( 1 - \chi_n^{p - 1} \psi_n \zeta^{q(d_np+1)} - \chi_n^{p - 2} \psi_n^2 \zeta^{q(d_n p + 2)}\right)
      \mod \langle \zeta^{q(d_n p+2)+2} \rangle.
    \end{displaymath}
    Together with~\eqref{eq:14} this implies
    \begin{displaymath}
      i_{n + 1}(f) = i_1(g) \ge q(d_n p + 1) = q d_{n + 1}
    \end{displaymath}
    and~\eqref{eq:13} with
    \begin{equation}
      \label{eq:15}
        \chi_{n + 1}
        =
        - \chi_n^{p - 1} \psi_n
        \text{ and }
        \psi_{n + 1}
        =
        - \chi_n^{p - 2} \psi_n^2.      
    \end{equation}
      This completes the proof of the induction step and of~\eqref{eq:13} for every integer~$n \ge 0$.
      Then the proposition follows from a direct computation using the recursion~\eqref{eq:15}, together with~\eqref{eq:10} and~\eqref{eq:11}.
\end{proof}

\subsection{Lower bound of the norm of periodic points}
\label{sec:lowerbound}
The goal of this section is to prove Theorem~\ref{thm:lower-bound}.
We first introduce some notation and recall a result from~\cite{LindahlRiveraLetelier2015}.

Let~$(\K, | \cdot |)$ be an ultrametric field, and recall that~$\mathcal{O}_{\K}$ denotes the ring of integers of~$\K$, and~$\mathfrak{m}_{\K}$ the maximal ideal of~$\mathcal{O}_{\K}$.
Denote the residue field of~$\K$ by~$\widetilde{\K} \= \mathcal{O}_{\K} / \mathfrak{m}_{\K}$, and for an element~$a$ of~$\mathcal{O}_{\K}$, denote by the~$\widetilde{a}$ its reduction in~$\widetilde{\K}$.
The reduction of a power series~$f$ in~$\mathcal{O}_{\K}[[\zeta]]$, is the power series~$\widetilde{f}$ in~$\widetilde{\K}[[\zeta]]$ whose coefficients are the reductions of the corresponding coefficients of~$f$.
For a power series~$f$ in $\mathcal{O}_{\K}[[\zeta]]$, the \emph{Weierstrass degree} $\wideg(f)$ of~$f$ is the order in $\widetilde{\K}[[\zeta]]$ of the reduction $\widetilde{f}$ of~$f$.
Note that if $\wideg(f)$ is finite, then the number of zeros of~$f$ in~$\mathfrak{m}_{\K}$, counted with multiplicity, is less than or equal to~$\wideg(f)$, see, \emph{e.g.}, \cite[\S{VI}, Theorem~9.2]{Lang2002}.

In the case the characteristic~$p$ of~$\widetilde{\K}$ is positive, and~$f$ is a wildly ramified power series in~$\mathcal{O}_{\K}[[\zeta]]$, it is well-known that the minimal period of every periodic point of~$f$ in~$\mathfrak{m}_{\K}$ is a power of~$p$.

\begin{mydef}
  Let $p$ be a prime number and $\K$ field of characteristic $p$.
  For a wildly ramified power series~$f$ in $\K[[\zeta]]$, define for each integer $n \geq 0$ the element~$\delta_n(f)$ of $\K$ as follows: Put $\delta_n(f) \= 0$ if $i_n(f) = +\infty$, and otherwise let $\delta_n(f)$ be the coefficient of~$\zeta^{i_n(f)+1}$ in~$f^{p^n}(\zeta)$.
\end{mydef}

\begin{lemma}[Special case of Lemma~2.4 in \cite{LindahlRiveraLetelier2015}]\label{lem24}
  Let~$p$ be a prime number and $(\K,|\cdot|)$ an ultrametric field of characteristic $p$.
  Then, for every wildly ramified power series~$f$ in $\mathcal{O}_{\K}[[\zeta]]$, the following properties hold.
\begin{enumerate}
\item
  Let~$w_0$ in $\mathfrak{m}_{\K}$ be a fixed point of~$f$ different from~$0$.
  Then we have
  \begin{displaymath}
    |w_0|\geq |\delta_0(f)|
  \end{displaymath}
  with equality if and only if
  \begin{displaymath}
    \wideg(f(\zeta)-\zeta)= i_0(f)+2.
  \end{displaymath}
\item
  Let $n\geq1$ be an integer and~$\zeta_0$ in~$\mathfrak{m}_{\K}$ a periodic point of~$f$ of minimal period~$p^n$.
  If in addition $i_n(f)<+\infty$, then we have
  \begin{displaymath}
    |\zeta_0|
    \geq
    \left|\frac{\delta_n(f)}{\delta_{n-1}(f)}\right|^{\frac{1}{p^n}},
  \end{displaymath}
  with equality if and only if
  \begin{equation}
    \label{widegzeta}
    \wideg\left(\frac{f^{p^n}(\zeta)-\zeta}{f^{p^{n-1}}(\zeta)-\zeta}\right) = i_n(f)-i_{n-1}(f)+p^n.
  \end{equation}
  Moreover, if (\ref{widegzeta}) holds, then the cycle containing $\zeta_0$ is the only cycle of minimal period $p^n$ of~$f$ in $\mathfrak{m}_{\K}$, and for every point $\zeta_0'$ in this cycle $|\zeta_0'|=\left|\frac{\delta_n(f)}{\delta_{n-1}(f)}\right|^{\frac{1}{p^n}}$.
\end{enumerate}
\end{lemma}

\begin{proof}[Proof of Theorem~\ref{thm:lower-bound}]
  The assertion about fixed points is a direct consequence of ${\delta_0(f) = a}$ and Lemma~\ref{lem24}(1).
  
  To prove the statement about periodic points that are not fixed, note first that this statement holds trivially in the case~$\resit(f) = 0$.
  Thus, we assume that~$\resit(f) \neq 0$, and therefore~$f$ is $q$-ramified by Theorem~\ref{thm:q-ramification}.
  In particular, for every integer~$n \ge 1$ we have~$i_n(f) < + \infty$.
  On the other hand, by Proposition~\ref{prop:deltaiterates} we have for every integer~$n \ge 1$
  \begin{displaymath}
    \delta_n(f) = a^{\frac{p^{n + 1} - 1}{p - 1}} \resit(f)^{\frac{p^n-1}{p-1}}.
  \end{displaymath}
  Hence, by Lemma~\ref{lem24}(2) we have for every periodic point~$\zeta_0$ in~$\mathfrak{m}_k$ of minimal period~$p^n$,
\begin{equation}\label{normbound2}
  |\zeta_0|
  \geq
  \left|\frac{\delta_n(f)}{\delta_{n-1}(f)}\right|^{\frac{1}{p^n}}
  =
  \left| a^{p^n} \resit(f)^{p^{n-1}} \right|^{\frac{1}{p^n}}
  =
  |a| \cdot | \resit(f)|^{\frac{1}{p}}.
\end{equation}
This completes the proof of Theorem~\ref{thm:lower-bound}.
\end{proof}

\begin{rem}
Equality in \eqref{normbound2} is, as seen in Lemma~\ref{lem24}, given by a condition on the reduction of $f$. In the case of equality, for $q$-ramified power series  \emph{all} periodic points in the open unit disk, which are not fixed by~$f$, in fact lie on the sphere about the origin of radius~$|\delta_0(f)| \cdot |\resit(f)|^{\frac{1}{p}}$, see Example~\ref{e:optimality} in~\S\ref{sec:furth-results-exampl}.
\end{rem}

\section{Proof of the Main Lemma}
\label{s:proof of Main Lemma}

The goal of this section is to prove the Main Lemma.
We use the strategy introduced in~\cite[\S3.2]{RiveraLetelier2003} and~\cite{LindahlRiveraLetelier2013}, using the power series~$(\Delta_m)_{m = 0}^{+ \infty}$ defined by~\eqref{eq:27} and~\eqref{eq:23}.
The proof is naturally divided into the cases~$q \le p - 1$ and~$q \ge p + 1$.

\partn{Case 1, $q \le p - 1$}
Note that in this case we have~$\ell = q$.
For each integer~$m \ge 1$ define~$\alpha_m$, $\beta_m$ and $\gamma_m$ in~$F_1$ by the recursive relations 
\begin{align}
\alpha_{m+1} & \= x_0 (qm+1)\alpha_m \label{receq1}\\ 
\beta_{m+1} & \= \left[ x_0^2 \binom{qm+1}{2} + x_1(qm+1)\right]\alpha_m + x_0 (q(m+1)+1)\beta_m \label{receq2}\\
\gamma_{m+1} & \= \left[ x_0^3 \binom{qm+1}{3} + x_0x_1 qm (qm+1)\right]\alpha_m \label{receq4}\\ 
             & \quad + \left[ x_0^2 \binom{q(m+1)+1}{2}+ x_1(q(m+1)+1)\right]\beta_m \notag
               \\ & \quad + x_0 (q(m+2)+1)\gamma_m, \notag
\end{align}
with initial conditions $\alpha_1 \= x_0$, $\beta_1 \= x_1$, and $\gamma_1 \= 0$.
We claim that for every integer~$m \ge 1$ we have
\begin{equation}
  \label{deltaclaim}
  \Delta_m(\zeta)
  \equiv
  \alpha_m\zeta^{qm+1} + \beta_m\zeta^{q(m+1)+1} + \gamma_m\zeta^{q(m+2)+1}\mod \langle \zeta^{q(m + 2) + 2} \rangle.
\end{equation}
For $m=1$ this holds by definition.
Assume this is valid for some $m\geq1$.
Then
\begin{displaymath}
  \begin{split}
  \Delta_{m+1}(\zeta)
  & = \Delta_m(\widehat{f}(\zeta)) - \Delta_m(\zeta)\\
&\equiv \alpha_m\zeta^{qm+1}\left[\left(1 + x_0 \zeta^q + x_1\zeta^{2q} + x_2\zeta^{3q+1} + \cdots\right)^{qm+1} - 1\right]\\
&\qquad + \beta_m\zeta^{q(m+1)+1}\left[\left(1 + x_0 \zeta^q + x_1\zeta^{2q} + x_2\zeta^{3q+1} + \cdots\right)^{q(m+1)+1}-1\right]\\
&\qquad + \gamma_m\zeta^{q(m+2)+1}\left[\left(1 + x_0 \zeta^q + x_1\zeta^{2q} + x_2\zeta^{3q+1} + \cdots\right)^{q(m+2) +1}-1\right]\\
&\qquad \mod \langle \zeta^{q(m+3)+2} \rangle\\
  &\equiv
  \alpha_m\left[ \zeta^{q(m+1)+1}x_0 (qm+1) + \zeta^{q(m+2)+1}\left(x_0^2 \binom{qm+1}{2} + x_1(qm+1)\right)
  \right. \\ &  \qquad \left. 
    + \zeta^{q(m+3)+1}\left( x_0^3 \binom{qm+1}{3} + x_0 x_1 qm (qm+1)\right) \right]
    \\
    &\qquad+ \beta_m\Bigg[ \zeta^{q(m+2)+1} x_0 (q(m+1)+1)
    \\ & \qquad 
    + \zeta^{q(m+3)+1}\left( x_0^2 \binom{q(m+1)+1}{2} + x_1(q(m+1)+1)\right) \Bigg]
\\
&\qquad + \gamma_m\zeta^{q(m+3)+1} x_0 (q(m+2)+1) \mod \langle \zeta^{q(m+3)+2} \rangle,    
  \end{split}
\end{displaymath}
which proves the induction step and~\eqref{deltaclaim}.

In view of~\eqref{eq:7} and~\eqref{deltaclaim}, to prove the Main Lemma with $q \le p - 1$, it is sufficient to prove
\begin{equation}
  \label{e:main alpha}
  \alpha_p \equiv 0 \mod p F_1,
\end{equation}
\eqref{e:main beta} with~$\beta = \beta_p$, and~\eqref{e:main gamma}~$\gamma = \gamma_p$.
The first~2 are given by Proposition~\ref{prop:deltaiteratesshort}, so we only need to prove the latter.
To do this, we solve~\eqref{receq4} explicitly, utilizing the explicit solutions of~\eqref{receq1} and~\eqref{receq2} given in the proof of Proposition~\ref{prop:deltaiteratesshort}.
Assume first~$q \equiv - 1 \pmod{p}$.
By~\eqref{alpham} and~\eqref{betam} with~$m = p - 1$, we have
\begin{displaymath}
  \alpha_{p - 1} \equiv 0 \mod p F_1
  \text{ and }
  \beta_{p - 1} \equiv - x_0^{p - 2} x_1 \mod p F_1.
\end{displaymath}
Combined with~\eqref{receq4} with~$m = p - 1$, this implies
\begin{displaymath}
  \gamma_p
  \equiv
  - x_0^{p - 2} x_1^2 \mod p F_1.
\end{displaymath}
This proves~\eqref{e:main gamma} with~$\gamma = \gamma_p$, when~$q \equiv - 1 \pmod{p}$.

It remains to prove~\eqref{e:main gamma} with~$\gamma = \gamma_p$, when~$q \not\equiv - 1 \pmod{p}$.
Denote by~$r_0$ the unique~$r$ in~$\{1, \ldots, p - 1 \}$ such that $qr \equiv - 1 \pmod{p}$.
By our assumption~$q \not\equiv - 1 \pmod{p}$, we have~$r_0 \neq 1$ and therefore
\begin{equation}
  \label{eq:4}
  r_0 \in \{2, \ldots, p - 1 \}.
\end{equation}
Noting that for every~$j \ge 0$ we have~$qj + 1 > 0$, we use the substitution
\[\gamma^*_m \=
\frac{\gamma_mx_0^2}{(q(m+1)+1)(qm+1)\alpha_m}. \]
Note that by~\eqref{alpham} we have
\[\gamma^*_m =
  \gamma_m\bigg/ \left( x_0^{m - 2} \prod_{j=1}^{m + 1} (qj+1) \right). \]
  On the other hand, by~\eqref{alpham} and~\eqref{betam} we get 
  \[\frac{\beta_m}{\alpha_m} = \frac{1}{x_0} \sum_{r=1}^m \left(x_0^2\frac{q(r-1)}{2} + x_1\right)\frac{qm+1}{qr+1}. 
  \]
  By plugging these equations into~\eqref{receq4}, we obtain
\begin{displaymath}
  \begin{split}
  \gamma^*_{m+1}
  & =
    \gamma^*_m + \frac{qm}{(q(m+1)+1)(q(m+2)+1)} x_0^2 \left( x_0^2 \frac{qm-1}{6} + x_1 \right)
  \\
  & \quad +\frac{1}{q(m+2)+1} \left( x_0^2 \frac{q(m+1)}{2} + x_1 \right)
      \sum_{r=1}^{m} \left(x_0^2 \frac{q(r - 1)}{2} + x_1 \right) \frac{1}{qr + 1}.    
  \end{split}
\end{displaymath}
Using~$\gamma^*_{1} = 0$ and defining for every integer~$s$
\[H(s) \= x_0^2 \frac{qs}{2}+x_1, \]
we obtain inductively for each $m\geq 1$
\begin{multline*}
  \gamma^*_{m}
  =
    \sum_{s=1}^{m-1} \left[ \frac{qs}{(q(s+1)+1)(q(s+2)+1)} x_0^2 \left(x_0^2 \frac{qs - 1}{6} + x_1 \right)
  \right. \\ \left.
    + \frac{H(s+1)}{q(s+2)+1}\sum_{r=1}^{s}\frac{H(r - 1)}{qr + 1} \right].
\end{multline*}
Equivalently,
\begin{multline}
  \label{dmexplicit}
    \gamma_{m}
    =
    x_0^{m - 2} \sum_{s=1}^{m-1}\left[ x_0^2 qs \left(x_0^2 \frac{qs - 1}{6} + x_1 \right) \prod_{j\in \{1,\ldots,m+1\}\setminus\{s+1,s+2\}}(qj+1)
  \right. \\ \left.
    + H(s+1)\sum_{r=1}^{s}H(r - 1) \prod_{j\in \{1,\ldots,m+1\}\setminus \{r, s+2\}}(qj+1)\right].
\end{multline}
Setting~$m = p$, for every~$s$ in~$\{1, \ldots, p - 1\}$ we have by Lemma~\ref{wilson}
\begin{displaymath} 
  \prod_{\substack{j\in \{1,\ldots,p+1\} \\ j \not\in \{s+1,s+2\}}}(qj+1)
  \equiv
  \begin{cases}
    - \frac{q(p+1)+1}{q(r_0 + 1)+1} \equiv -\frac{q+1}{q} \mod p\Z_{(p)}
    & \text{if~$s = r_0 - 1$};
    \\
    - \frac{q(p+1)+1}{q(r_0 - 1)+1} \equiv \frac{q+1}{q} \mod p\Z_{(p)}
    & \text{if $s = r_0 - 2$};
    \\
    0
    & \text{otherwise}.
  \end{cases}
\end{displaymath}
Analogously, for every~$s$ in~$\{1, \ldots, p - 1 \}$ and~$r$ in~$\{1, \ldots, s\}$, we have
\begin{displaymath} 
  \prod_{\substack{j\in \{1,\ldots,p+1\} \\ j \not\in \{r, s+2\}}} (qj+1)
  \equiv
  \begin{cases}
    - \frac{q+1}{qr + 1} \mod p\Z_{(p)}
    & \text{if $s = r_0 - 2$};
    \\
    - \frac{q+1}{q(s + 2) + 1} \mod p\Z_{(p)}
    & \text{if~$s \ge r_0$ and~$r = r_0$};
    \\
    0
    & \text{otherwise}.
  \end{cases}
\end{displaymath}
Combined with~\eqref{dmexplicit} with~$m = p$ and
\begin{equation}
  \label{eq:2}
  H(r_0 - 1) \equiv - x_0^2 \frac{q + 1}{2} + x_1 \mod p \Z_{(p)},
\end{equation}
these congruences imply
\begin{equation}
  \label{eq:3}
  \begin{split}
  \gamma_p
  & \equiv
  - x_0^{p} q(r_0 - 1) \left( x_0^2 \frac{q(r_0 - 1) - 1}{6} + x_1 \right) \frac{q + 1}{q}
  \\ & \quad
  + x_0^{p} q (r_0 - 2) \left( x_0^2 \frac{q (r_0 - 2) - 1}{6} + x_1 \right) \frac{q + 1}{q}
  \\ & \quad
  - x_0^{p - 2} H(r_0 - 1) \sum_{r = 1}^{r_0 - 2} H(r - 1) \frac{q + 1}{qr + 1}
  \\ & \quad
  - x_0^{p - 2} \sum_{s = r_0}^{p - 1} H(s + 1)H(r_0 - 1) \frac{q + 1}{q(s + 2) + 1} \mod p F_1
\\ & \equiv
- x_0^{p} (q + 1) H(r_0 - 1)
\\ & \quad
- x_0^{p - 2} (q + 1) H(r_0 - 1)
\sum_{\substack{r \in \{1, \ldots, p + 1 \} \\ r \not\in \{ r_0 - 1, r_0, r_0 + 1 \}}} \frac{H(r - 1)}{q r + 1}
\mod p F_1.
\end{split}
\end{equation}
By~\eqref{eq:4}, we have
\begin{multline}
  \label{eq:5}
  \sum_{\substack{r \in \{1, \ldots, p + 1 \} \\ r \not\in \{ r_0 - 1, r_0, r_0 + 1 \}}} \frac{H(r - 1)}{q r + 1}
  \\
  \begin{aligned}
    & \equiv
    \sum_{\substack{r \in \{1, \ldots, p + 1 \} \\ r \not\in \{ r_0 - 1, r_0, r_0 + 1 \}}} \left( \frac{x_0^2}{2} + \frac{H(r_0 - 1)}{q r + 1} \right)
\mod p F_1
  \\ & \equiv
- x_0^2 + H(r_0 - 1) \sum_{\substack{r \in \{1, \ldots, p + 1 \} \\ r \not\in \{ r_0 - 1, r_0, r_0 + 1 \}}} \frac{1}{q r + 1}
\mod p F_1.
  \end{aligned}
\end{multline}
On the other hand, by the second assertion of Lemma~\ref{wilson}, we have
\begin{multline}
  \label{eq:25}
  \sum_{\substack{r \in \{1, \ldots, p + 1 \} \\ r \not\in \{ r_0 - 1, r_0, r_0 + 1 \}}} \frac{1}{q r + 1}
  \\
  \begin{aligned}
  & \equiv
  \frac{1}{q(p + 1) + 1} - \frac{1}{q(r_0 - 1) + 1}  - \frac{1}{q(r_0 + 1) + 1}
  \mod p \Z_{(p)}
  \\ & \equiv
  \frac{1}{q + 1}
  \mod p \Z_{(p)}.
  \end{aligned}
\end{multline}
Together with~\eqref{eq:2}, \eqref{eq:3}, and~\eqref{eq:5}, this implies~\eqref{e:main gamma} with~$\gamma = \gamma_p$ and completes the proof of the Main Lemma in the case~$q \le p - 1$.

\partn{Case 2, $q \ge p + 1$}
Note that in this case our hypotheses on~$\ell$ imply in all the cases that~$q \ge 2 \ell + 1$.
For each integer~$m \ge 1$ define~$\al_m$, $\be_m$ and $\ga_m$ in~$F_1$ by the recursive relations 
\begin{align}
& \al_{m+1} \= x_0(qm+1)\al_m \label{receq1b}\\ 
&\be_{m+1} \= x_1(qm+1)\al_m + x_0(qm+\ell+1)\be_m \label{receq2b}\\
& \ga_{m+1} \= x_1(qm +\ell+1)\be_m + x_0(qm+2\ell+1)\ga_m\label{receq4b},
\end{align}
with initial conditions $\al_1 \= x_0$~$\be_1 \= x_1$, and $\ga_1 \=0$.
We claim that for every integer~$m \ge 1$ we have
\begin{equation}
    \label{deltaclaimhat}
  \Delta_m(\zeta)
  \equiv
  \al_m\zeta^{qm+1} + \be_m\zeta^{qm+\ell+1}  + \ga_m\zeta^{qm+2\ell+1}\mod \langle \zeta^{qm+2\ell+2} \rangle.
\end{equation}
For $m=1$ this holds by definition.
Assume further this is valid for some $m\geq1$.
Then, using $q \ge 2 \ell + 1$, we have
\begin{displaymath}
  \begin{split}
\Delta_{m+1}(\zeta)
& = \Delta_{m}(\widehat{f}(\zeta)) - \Delta_{m}(\zeta)\\
& \equiv \al_m\zeta^{qm+1}\left[\left(1 + x_0\zeta^{q} + x_1\zeta^{q+\ell}   + x_2\zeta^{q+2\ell + 1} + \cdots\right)^{qm+1} - 1\right]\\
&\quad + \be_m\zeta^{qm+\ell+1}\left[\left(1 +x_0\zeta^{q} + x_1\zeta^{q+\ell}   + x_2\zeta^{q+2\ell + 1} + \cdots\right)^{qm + \ell + 1}-1\right]\\
&\quad + \ga_m\zeta^{qm + 2\ell + 1}\left[\left(1 + x_0\zeta^{q} + x_1\zeta^{q+\ell}   + x_2\zeta^{q+2\ell + 1} + \cdots\right)^{qm + 2\ell + 1}-1\right]\\
&\quad \mod \langle \zeta^{q(m + 1) + 2\ell + 2} \rangle
\\
&\equiv\al_m\left(\zeta^{q(m+1) + 1}x_0(qm+1) + \zeta^{q(m+1)+\ell+1}x_1(qm+1)\right) \\
&\quad+ \be_m\left(\zeta^{q(m+1)+\ell+1}x_0(qm+\ell+1) + \zeta^{q(m+1)+2\ell+1}x_1(qm+\ell+1)\right)
\\
& \quad+ \ga_m\zeta^{q(m+1)+2\ell+1}x_0(qm+2\ell+1) \mod \langle \zeta^{q(m+1) + 2\ell +2} \rangle,    
  \end{split}
\end{displaymath}
which proves the induction step and the claim \eqref{deltaclaimhat}.

In view of~\eqref{eq:7} and~\eqref{deltaclaimhat}, to complete the proof of the Main Lemma in the case~$q \ge p + 1$, it is sufficient to prove
\begin{equation}
  \label{e:main alpha bis}
  \al_p \equiv 0 \mod p F_1,
\end{equation}
\eqref{e:main beta} with~$\beta = \be_p$, and~\eqref{e:main gamma} with~$\gamma = \ga_p$.
The linear recursion described in \eqref{receq1b}, \eqref{receq2b} and \eqref{receq4b} can be solved explicitly.
By telescoping \eqref{receq1b}, we obtain for every~$m \ge 1$ the solution
\begin{equation}
  \label{alphamhat}
  \al_m
  =
  x_0^{m}\prod_{j=1}^{m-1}(qj+1).
\end{equation}
Taking~$m = p$, this implies~\eqref{e:main alpha bis}.

On the other hand, inserting \eqref{alphamhat} in~\eqref{receq2b} yields
\[\be_{m+1} = x_0^{m}x_1\prod_{j=1}^m (qj+1) + x_0(qm+\ell+1)\be_m.\]
Then, an induction argument shows that for every~$m \ge 1$ we have
\begin{equation}\label{betamhat}
  \be_m \equiv x_0^{m-1}x_1\sum_{r=1}^{m}\prod_{j \in \{1, \ldots, m\} \setminus \{ r \}} (q j+1) \mod p F_1.
\end{equation}
When~$m = p$ every term in the sum above contains a factor~$p$, except for the unique~$r_0$ in~$\{1, \ldots, p - 1 \}$ satisfying $qr_0 \equiv -1 \pmod{p}$.
Then by Lemma~\ref{wilson}, we have
\begin{displaymath}
  \begin{split}
  \be_p
  &\equiv
    x_0^{p-1}x_1\prod_{j \in \{ 1, \ldots, p \} \setminus \{ r_0 \}} (q j+1) \mod p F_1
  \\ &\equiv
       -x_0^{p-1}x_1 \mod p F_1.    
  \end{split}
\end{displaymath}
This proves~\eqref{e:main beta} with~$\beta = \be_p$.

To prove~\eqref{e:main gamma} with~$\gamma = \ga_p$, assume first~$q \equiv - 1 \pmod{p}$.
Then by~\eqref{receq4b} with~$m = p - 1$, \eqref{betamhat}, and Lemma~\ref{wilson} we have
\begin{displaymath}
  \begin{split}
  \ga_p
  & \equiv
  x_1 \be_{p - 1} \mod p F_1
  \\ & \equiv
  x_0^{p - 2} x_1^2 \sum_{r = 1}^{p - 1} \prod_{j \in \{1, \ldots, p - 1\} \setminus \{ r \}} (q j + 1) \mod p F_1
  \\ & \equiv
       x_0^{p - 2} x_1^2 \prod_{j \in \{2, \ldots, p - 1\}} (1 - j) \mod p F_1
  \\ & \equiv
       - x_0^{p - 2} x_1^2 \mod p F_1.    
  \end{split}
\end{displaymath}
It remains to prove~\eqref{e:main gamma} with~$\gamma = \ga_p$ in the case~$q \not\equiv - 1 \pmod{p}$.
Note that in this case $r_0 \neq 1$.
Inserting~\eqref{betamhat} in~\eqref{receq4b}, we obtain
\begin{multline}
  \label{eq:9}
  \ga_{m+1} \equiv x_0^{m-1}x_1^2\sum_{r = 1}^{m}\prod_{j \in \{1, \ldots, m+1\} \setminus \{ r \}} (q j+1)
  \\
       + x_0(q(m+2)+1)\ga_m \mod p F_1.
     \end{multline}
     For every~$m \ge 1$ define~$\cgamma_m$ in~$F_1$ recursively, by~$\cgamma_1 \= 0$ and for~$m \ge 1$, by
     \begin{equation}
       \label{eq:1}
       \cgamma_{m + 1}
       \=
       x_0^{m-1}x_1^2\sum_{r=1}^{m}\prod_{j \in \{1, \ldots, m+1\} \setminus \{ r \}} (q j+1)
       + x_0(q(m+2)+1)\cgamma_m.
     \end{equation}
     Note that by~\eqref{eq:9} for every~$m \ge 1$ we have~$\cgamma_m \equiv \ga_m \mod p F_1$.
Using that for every~$j \ge 0$ we have~$q j + 1 > 0$, and the substitution
\[\cgamma^*_m
  \=
  \cgamma_m \bigg/ \left(x_0^{m-2}x_1^2\prod_{j=1}^{m+1} (q j+1) \right),\]
we obtain
\[\cgamma_{m+1}^*
  =
  \cgamma_{m}^* + \frac{1}{q(m+2)+1}\sum_{r=1}^{m}\frac{1}{q r + 1}.\]
Inductively we have
\begin{equation}
  \label{eq:8}
  \cgamma_m^*
  =
    \sum_{s=1}^{m-1}\frac{1}{q (s+2) + 1}\sum_{r=1}^{s}\frac{1}{q r + 1},
  \end{equation}
  which is a rational number.
Since~$r_0 \neq 1$, for every~$r$ in~$\{1, \ldots, p + 1 \} \setminus \{ r_0 \}$ we have that~$\frac{1}{q r + 1}$ is in~$\Z_{(p)}$.
Thus, taking~$m = p$ in~\eqref{eq:8}, and using~\eqref{eq:25}, we obtain
\begin{equation*}
  \begin{split}
    (q r_0 + 1) \cgamma_p^*
    & \equiv
    \sum_{\substack{r \in \{1, \ldots, p + 1 \} \\ r \not\in \{ r_0 - 1, r_0, r_0 + 1 \}}} \frac{1}{q r + 1} \mod p \Z_{(p)}
      \\ & \equiv
      \frac{1}{q + 1} \mod p \Z_{(p)}.
  \end{split}
\end{equation*}
Using Lemma~\ref{wilson}, we obtain
\begin{displaymath}
  \begin{split}
  \cgamma_p
  & \equiv
  x_0^{p-2}x_1^2 \frac{1}{q + 1} \prod_{j \in \{1, \ldots, p + 1 \} \setminus \{ r_0 \}} (q j + 1) \mod p F_1
  \\ & \equiv
    - x_0^{p-2}x_1^2  \mod p F_1.    
  \end{split}
\end{displaymath}
This completes the proof of~\eqref{e:main gamma} with~$\gamma = \ga_p$ and of the Main Lemma.

\section{Further results and examples}
\label{sec:furth-results-exampl}

In this section we gather several examples illustrating our results and state some further consequences of our main theorems.

\begin{example}
  \label{ex:p + 1}
  The following example shows that the conclusion of Theorem~\ref{thm:q-ramification} is false when~$q = p + 1$ and~$p$ is odd.
  Consider the polynomial with coefficients in~$\F_p$,
  \begin{displaymath}
    P(\zeta)
    \=
    \zeta (1 + \zeta^{p + 1} + \zeta^{p + 2} + \zeta^{2(p + 1)}).
  \end{displaymath} 
  A direct computation using~\eqref{eq:genericity polynomial} shows that~$\resit(P) = 1$.
  On the other hand, using the Main Lemma with $q = p + 1$, $\ell = 1$, and $x_0=x_1=1$, we have
  \begin{displaymath}
    i_1(P) = p^2 + p + 1
    <
    i_0(P)(p + 1),
  \end{displaymath}
  so~$P$ is not $(p + 1)$-ramified.
\end{example}

There is another natural source of power series~$f$ that satisfy~$i_0(f) = p + 1$ and that are not~$(p + 1)$-ramified.
Let~$g$ in~$\K[[\zeta]]$ be a~$1$-ramified power series, and put~${f \= g^p}$.
Then
\begin{displaymath}
  i_0(f)
  =
  i_1(g)
  =
  p+1
  \text{ and }
  i_1(f) = i_2(g) = 1+p+p^2
  <
  i_0(f) (p+1),
\end{displaymath}
so~$f$ is not $(p+1)$-ramified.
For concreteness, let~$a$ in~$\K$ be different from~$1$, and assume that~$g$ is of the form
\[ 
  g(\zeta)
  \equiv
  \zeta(1 + \zeta + a \zeta^2) \mod \langle \zeta^4 \rangle.
\]
In view of~\eqref{eq:genericity polynomial}, we have~$\resit(g) = 1 - a \neq 0$, so~$g$ is~$1$-ramified by Theorem~\ref{thm:q-ramification}.
On the other hand, for $p = 3$, $5$ and $7$ a computation shows that $\resit(f) = (1 - a)^p \neq 0$.
Thus, in contrast with the situation for~$q$ in~$\{1, \ldots, p - 1 \}$ in Theorem~\ref{thm:q-ramification}, for~$q = p + 1$ and~$p = 3$, $5$ and $7$ the nonvanishing of the iterative residue does not imply $(p + 1)$\nobreakdash-ramification.\footnote{The situation is now clear form the recent characterization of $(p + 1)$-ramification by the first named author in~\cite{Nor1909}.}
So, the following question arises naturally.

\begin{question}
  For which $1$-ramified power series~$g$ in~$\K[[\zeta]]$ do we have~$\resit(g^p) \neq 0$?
\end{question}

\begin{example}
  \label{ex:p - 1}
  The following example illustrates Theorem~\ref{thm:q-ramification} in the case~$q = p - 1$.
  A direct computation shows that for the polynomial~$P(\zeta) \= \zeta + \zeta^p$, we have for every integer~$n \ge 1$
  \begin{displaymath} 
    P^{p^n}(\zeta) = \zeta + \zeta^{p^{p^n}}.
  \end{displaymath}
  In particular, $i_n(P) = p^{p^n} - 1$, and therefore~$P$ is not $(p - 1)$-ramified.
  This is consistent with Theorem~\ref{thm:q-ramification}, since by Theorem~\ref{thm:closed-formula} we have~$\resit(P) = \ind(P) = 0$.
\end{example}

\begin{example}
  \label{e:optimality}
  This example shows that the lower bound~\eqref{normbound} in Theorem~\ref{thm:lower-bound} is optimal for~$p \ge 5$ and~$q \le p - 3$.
  Let~$p \ge 3$ be a prime number, ($\K$, $| \cdot |$) an ultrametric field of characteristic~$p$, and~$q$ in~$\{1, \ldots, p - 1 \}$.
  Furthermore, let~$a$ and~$b$ in~$\K$ be such that $0 < |a| < 1$ and $|b|=1$, and let~$f$ be a power series in~$\K[[z]]$ satisfying
  \begin{displaymath}
    f(\zeta) \equiv \zeta(1 + a\zeta^q + b\zeta^{q+1}) \mod \langle \zeta^{2q + 4} \rangle.
  \end{displaymath}
  A direct computation using~\eqref{eq:genericity polynomial} shows that
  \begin{displaymath}
    \resit(f) = \frac{q + 1}{2} + (-1)^q \frac{b^q}{a^{q + 1}} \neq 0,
  \end{displaymath}
  so by Theorem~\ref{thm:q-ramification} the series~$f$ is $q$-ramified.
  In the case~$q \le p - 2$, by~\eqref{eq:genericity polynomial} the reduction~$\widetilde{f}$ of~$f$ satisfies~$\resit(\widetilde{f}) = \frac{q + 2}{2}$.
  Assuming further that~$q \le p - 3$, we have~$\resit(\widetilde{f}) \neq 0$, and we obtain that~$\widetilde{f}$ is ${(q+1)}$\nobreakdash-ramified by Theorem~\ref{thm:q-ramification}.
  This implies that~\eqref{widegzeta} in Lemma~\ref{lem24} holds for every integer~$n \ge 1$.
  It follows that for every periodic point~$\zeta_0$ of~$f$ in~$\mathfrak{m}_{\K}$ that is not fixed, we have
  \[|\zeta_0| = |a| \cdot |\resit(f)|^{\frac{1}{p}}, \]
  see the proof of Theorem~\ref{thm:lower-bound}.
\end{example}

The following result is a direct consequence of Theorems~\ref{thm:q-ramification} and~\ref{thm:lower-bound} for fixed points whose multiplier is a root of unity, compare with~\cite[Corollary~C]{LindahlRiveraLetelier2015}.

\begin{corr}
  \label{c:higher order}
  Let $\K$ be an ultrametric field of odd characteristic, let $\gamma$ in~$\K$ be a root of unity, and denote by~$q \ge 1$ the order of~$\gamma$.
  Moreover, let~$f$ be a power series with coefficients in~$\K$ satisfying $f(0)=0$ and $f'(0)=\gamma$.
  If
  \begin{displaymath}
    q' \= \mult(f^q) - 1 \le p - 1
    \text{ and }
    \resit(f^q) \neq 0,
  \end{displaymath}
  then $f^q$ is $q'$-ramified.
  In particular, if~$f$ converges on a neighborhood of the origin, then the origin is isolated as a periodic point of~$f$.
\end{corr}

\begin{example}
  Let~$\K$ be an ultrametric field of characteristic~$7$, and note that~$2$ is a root of unity in~$\K$ of order~$3$.
  Let~$f$ be a power series with coefficients in~$\mathcal{O}_{\K}$ such that
  \begin{displaymath}
    f(\zeta)
    \equiv
    2\zeta + \zeta^2 \mod \langle \zeta^{13} \rangle.
  \end{displaymath}
  A direct computation shows that
  \begin{displaymath}
    f^3(\zeta) \equiv \zeta(1 + \zeta^6 + \zeta^7) \mod \langle \zeta^{13} \rangle.
  \end{displaymath}
  In particular, $\mult(f^3) - 1 = 6 > 3$, so~$f$ is not minimally ramified in the sense of~\cite{LindahlRiveraLetelier2015}, and we cannot apply Corollary~C of that paper to~$f$.
  However, by~\eqref{eq:genericity polynomial} we have $\resit(f^3) \neq 0$, so Corollary~\ref{c:higher order} applies to~$f^3$ and it implies that~$f^3$ is $6$\nobreakdash-ramified and that the origin is isolated as a periodic point of~$f^3$, and hence of~$f$.
\end{example}

\appendix

\section{Iterative residue in positive characteristic}\label{app:resit}
In this section we study the behavior of the iterative residue under iteration, which is defined for a power series~$f$ with coefficients in a field of characteristic different from~$2$, by~\eqref{eq:17}.
For a ground field of characteristic zero, this behavior can be understood from a relatively easy computation using the normal form~\eqref{eq:22}.\footnote{See also \cite[Lemma~12.9]{Mil06c} for a different approach for convergent power series.}
For a ground field of positive characteristic, not every power series~$f$ is formally conjugated to~\eqref{eq:22}, so we cannot apply this strategy.
We use instead the closed formula for the residue fixed point index~\eqref{eq:closed-formula} in Theorem~\ref{thm:closed-formula}.

\begin{prop}
  \label{resitcharp}
  Let~$\K$ a field of characteristic different from~$2$, and let~$f$ be a power series with coefficients in~$\K$ such that
  \begin{displaymath}
    f(0) = 0,
    f'(0) = 1
    \text{ and }
    f(z) \neq z.
  \end{displaymath}
  Then, for every integer~$n \ge 1$ that is not divisible by the characteristic of~$\K$, we have
  \begin{equation}
    \label{eq:16}
    \resit(f^n)
    =
    \frac{1}{n}\resit(f).
  \end{equation}
\end{prop}

For a field of characteristic~$2$, the formula~\eqref{eq:17} defining the iterative residue is meaningless.
Instead, we study the behavior of the residue fixed point index under iteration.
\begin{prop}
  \label{resitchar2}
  Let $\K$ be a field of characteristic~$2$, and let~$f$ be a power series with coefficients in~$\K$ such that~$q \= \mult(f) - 1 \ge 1$.
  Then, for every odd integer~$n \ge 1$ we have 
  \[\ind(f^n)
    =
    \begin{cases}
      \ind(f) +1
      & \text{if $q$ is even and $n\equiv 3 \pmod{4}$};
      \\
      \ind(f)
      & \text{otherwise.}
    \end{cases}
  \]
\end{prop}

The proofs of Proposition~\ref{resitcharp} and~\ref{resitchar2} are given after the following lemma.
For a field~$\K$ of positive characteristic, and an integer~$n \ge 0$, we use~$\binom{n}{2}$ to denote the reduction of this integer in the prime field of~$\K$.

\begin{lemma}\label{fniter}
  Let $\K$ be a field, let~$f$ be a power series with coefficients in~$\K$ such that~$q \= \mult(f) - 1 \ge 1$, and denote by~$a$ the coefficient of~$z^{q + 1}$ in~$f(z)$.
  Then, for every integer~$n \ge 1$ we have
  \begin{equation}
    \label{fncharp}
    f^n(z) -z\equiv n (f(z)-z) + \binom{n}{2}(q+1)a^2z^{2q+1} \mod \langle z^{2q+2} \rangle.
  \end{equation}
\end{lemma}

\begin{proof}
  We proceed by induction.
  The lemma holds trivially for $n=1$.
  Assume that \eqref{fncharp} holds for an integer~$n\geq1$.
  Put~$\Phi(z) \= \frac{f(z)-z}{z}$ and note that
  \begin{displaymath}
    \Phi(z)
    \equiv
    a z^q \mod \langle z^{q + 1} \rangle,
    \text{ and }
    \Phi(f(z))
    \equiv
    \Phi(z) + q a^2 z^{2q} \mod \langle z^{2q + 1} \rangle.
  \end{displaymath}
  Together with the induction hypothesis, this implies
\begin{displaymath}
  \begin{split}
  f^n\circ f(z)
  & \equiv
     f(z) + n f(z) \Phi(f(z)) + \binom{n}{2}(q+1)a^2f(z)^{2q+1} \mod \langle z^{2q+2} \rangle\\ 
  &\equiv
    z + z \Phi(z) + n z (1 + \Phi(z))\left(\Phi(z) + q a^2z^{2q}\right)\\
              &\qquad
                + \binom{n}{2}(q+1)a^2 z^{2q + 1} \mod \langle z^{2q+2} \rangle\\
&\equiv z + (n+1) z \Phi(z) + \left(n + \binom{n}{2}\right) (q+1) a^2z^{2q+1}\mod \langle z^{2q+2} \rangle\\
  &\equiv z + (n+1) (f(z) - z) + \binom{n+1}{2} (q+1) a^2z^{2q+1} \mod \langle z^{2q+2} \rangle.
    \qedhere    
  \end{split}
\end{displaymath}
\end{proof}

Given a field~$\K$, an integer~$q \ge 1$, and~$a_q$, \ldots, $a_{2q}$ in~$\K$, denote by~$P_q(a_q,\ldots,a_{2q})$ the right-hand side of~\eqref{eq:closed-formula}.
Note that for every~$\lambda$ in~$\K$ we have
\begin{equation}
  \label{eq:21}
  P_q( a_q, \ldots, a_{2q} + \lambda a_q^2)
  =
  P_q(a_q,\ldots, a_{2q}) + \lambda.
\end{equation}
If in addition~$\lambda$ is nonzero, then we also have
\begin{equation}
  \label{eq:18}
  P_q( \lambda a_q, \ldots, \lambda a_{2q})
  =
  \frac{1}{\lambda} P_q(a_q,\ldots, a_{2q}).
\end{equation}

\begin{proof}[Proof of Propositions~\ref{resitcharp} and~\ref{resitchar2}]
  Put
  \[f(z) = z(1 + a_q z^q + \cdots + a_{2q}z^{2q} + \cdots), \]
  so that~$a_q \neq 0$.
  A direct computation shows that for every integer~$n \ge 1$, we have
  \begin{displaymath}
    f^n(z)
    \equiv
    z(1 + n a_q z^q) \mod \langle \zeta^{q + 2} \rangle.
  \end{displaymath}
  In particular, if~$n$ is not divisible by the characteristic of~$\K$, then~$\mult(f^n) = q + 1$.
  On the other hand, by Theorem~\ref{thm:closed-formula},  Lemma~\ref{fniter}, \eqref{eq:21}, and~\eqref{eq:18}, we have
  \begin{equation}
    \begin{split}
    \label{eq:20}
    \ind(f^n)
  & =
  P_q\left(na_q, \ldots, na_{2q-1}, na_{2q} + \binom{n}{2} (q+1) a_q^2 \right)
  \\ & =
  \frac{1}{n} P_q\left(a_q, \ldots, a_{2q-1}, a_{2q} \right) + \frac{1}{n^2} \binom{n}{2} (q+1)
  \\ & = \frac{1}{n} \left[ \ind(f) + \frac{1}{n} \binom{n}{2} (q+1) \right].      
    \end{split}
  \end{equation}
If the characteristic of~$\K$ is different from~$2$, then by the definition of the iterative residue~\eqref{eq:17} we have
\begin{multline*}
  n \resit(f^n)
  =
  n \frac{\mult(f^n)}{2} - n\ind(f^n)
  \\ =
  n \frac{q + 1}{2} - \ind(f) - \frac{n - 1}{2} (q + 1)
  =
  \resit(f).
\end{multline*}
This proves Proposition~\ref{resitcharp}.
In the case the characteristic of~$\K$ is~$2$, Proposition~\ref{resitchar2} follows from~\eqref{eq:20} and from the fact that, in~$\K$, we have~$n = 1$ and
\begin{displaymath}
  \binom{n}{2}(q + 1)
  =
    \begin{cases}
      1
      & \text{if $q$ is even and $n\equiv 3 \pmod{4}$;}
      \\
      0
      & \text{otherwise.}
    \end{cases}
    \qedhere
\end{displaymath}
\end{proof}

\bibliographystyle{alpha} 

\end{document}